\documentclass[final,onefignum,onetabnum]{siamart190516}

\usepackage{amsmath,amssymb}
\RequirePackage[numbers]{natbib}
\RequirePackage{hyperref}
\usepackage[utf8]{inputenc} %unicode support
\usepackage{amsfonts}
\usepackage{bm}
\usepackage{cleveref}
\usepackage{graphicx}
\usepackage[export]{adjustbox}
\usepackage[caption=false,font=footnotesize,indention=0cm]{subfig}
\usepackage{pgfplots}
\usepackage{pgfplotstable}
\usepgfplotslibrary{external}
\usepackage[shortlabels]{enumitem}
\usepackage{algorithm}
\usepackage{algpseudocode}

\theoremstyle{definition}

\newtheorem{remark}[theorem]{Remark}

\def\N{\mathbb{N}}
\def\R{\mathbb{R}}
\def\C{\mathbb{C}}
\def\V{\mathcal{V}}

\def\addlegendimage{\csname pgfplots@addlegendimage\endcsname}

\ifpdf
\hypersetup{
  pdftitle={Deterministic sketching for Krylov subspace methods},
  pdfauthor={K. Bergermann}
}
\fi
\headers{Deterministic sketching for Krylov methods}{K.\ Bergermann}

\title{Deterministic sketching for Krylov subspace methods}% via row subset selection}

\author{Kai Bergermann\thanks{Centro di Ricerca Matematica Ennio De Giorgi, Scuola Normale Superiore, Pisa, Italy
		(\email{kai.bergermann@sns.it})}
}

\begin{document}

\maketitle

% REQUIRED
\begin{abstract}
Randomized sketching is currently introduced into every area of numerical linear algebra.
In Krylov subspace methods, it allows runtime savings at the cost of small accuracy reductions.
This work offers a different view on sketching in Krylov subspace methods by analyzing what subspace embeddings are obtained by arbitrary sketching matrices.
The analysis gives rise to a deterministic sketching approach leveraging row subset selection techniques that yield subspace embeddings holding with probability 1.
We propose deterministically sketched Krylov subspace methods for matrix functions, linear systems, and eigenvalue problems that show a similar performance to their randomly sketched counterparts.
\end{abstract}

% REQUIRED

% REQUIRED

\section{Introduction}\label{sec:intro}

Krylov subspace methods are a cornerstone of numerical linear algebra.
Over the last decades, a wealth of efficient iterative methods has been developed and analyzed for finding approximate solutions to fundamental linear algebra problems including linear systems, eigenvalue problems, matrix functions, and matrix equations in these low-dimensional subspaces \cite{saad2003iterative,saad2011numerical,golub2013matrix,liesen2013krylov,simoncini2016computational}.

A recent trend in numerical linear algebra is randomized sketching \cite{halko2011finding,woodruff2014sketching,martinsson2020randomized,epperly2026randomized}.
Its basic idea is to draw a random sketching matrix $\bm{S}\in\C^{s\times n}$ to reduce a problem of the large matrix $\bm{A}\in\C^{n\times k}$ to that of the sketched matrix $\bm{SA}\in\C^{s\times k}$ with $s\ll n$ by forming random linear combinations of the rows of $\bm{A}$.
Initial success of this approach occurred in least-squares problems \cite{sarlos2006improved,rokhlin2008fast} and low-rank approximation \cite{sarlos2006improved,rokhlin2010randomized,halko2011finding}; however, many recent works have focused on applying randomized sketching to Krylov subspace methods, cf., e.g., \cite{balabanov2022randomized,guttel2023randomized,cortinovis2024speeding,nakatsukasa2024fast,burke2024krylov,guttel2024sketch,bucci2025randomized,burke2025gmres,chung2025randomized,de2025arandomized,de2025brandomized,de2025crandomized,guidotti2025accelerating,jang2025randomized,krieger2025general,palitta2025sketched,palitta2025sketchedmateq}.
The goal of the latter is reducing computational costs of the full orthogonalization of the Krylov basis for non-symmetric problems by requiring merely orthogonality with respect to the $s$-dimensional inner product induced by the sketching matrix $\bm{S}$.
This $\bm{S}$-orthogonality can either be prescribed during the construction of a Krylov basis \cite{balabanov2022randomized,cortinovis2024speeding} or \emph{a-posteriori} via basis whitening after having constructed a non-orthogonal Krylov basis by an incomplete orthogonalization procedure such as the $k$-truncated Arnoldi method \cite{guttel2023randomized,nakatsukasa2024fast}.

Randomized algorithms have both a long history and many fields of application.
One such example is the Girard--Hutchinson estimator for approximating the trace of a matrix function by sums of quadratic forms with random Gaussian or Rademacher vectors \cite{girard1989fast,hutchinson1989stochastic}.
While effective for obtaining low-accuracy estimates, the convergence of this method is typically rather slow.

This work is motivated by ideas to derandomize the Girard--Hutchinson estimator by replacing random probing vectors by deterministic Hadamard \cite{bekas2007estimator} or graph coloring probing vectors \cite{tang2012probing,frommer2021analysis}.
These can lead to convergence of trace (or diagonal) estimators with fewer probing vectors by exploiting known problem properties such as Kronecker product structure \cite{bergermann2022fast} or entry decay in matrix functions \cite{benzi1999bounds,benzi2007decay,benzi2013decay}.
\Cref{fig:trace_estimation} shows the convergence behavior of the Girard--Hutchinson estimator for approximating the Estrada index \cite{estrada2000characterization} of the street network of Luxembourg\footnote{\url{https://sparse.tamu.edu/DIMACS10/luxembourg_osm}}.
It illustrates that investing upfront computations to obtain deterministic graph coloring probing vectors can lead to significantly faster convergence in comparison to the classical randomized probing approach\footnote{It is possible to combine probing and randomization in trace estimation and a similar combination for sketched Krylov methods would be an interesting road for future research.}.

\begin{figure}
	\centering
		\includegraphics[width=0.7\textwidth]{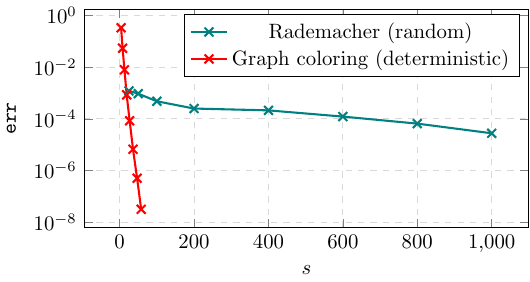}
	\caption{Relative trace estimation errors $\texttt{err} = \frac{|\text{tr}(e^{\beta\bm{A}}) - \frac{1}{s}\sum_{i=1}^s \bm{v}_i^\ast e^{\beta\bm{A}}\bm{v}_i |}{\text{tr}(e^{\beta\bm{A}})}$ for randomized and deterministic probing vectors $\bm{v}_i$ in the Girard--Hutchinson trace estimator $\text{tr}(e^{\beta\bm{A}}) \approx \frac{1}{s}\sum_{i=1}^s \bm{v}_i^\ast e^{\beta\bm{A}}\bm{v}_i$, where $\bm{A}\in\R^{114\, 599\times 114\, 599}$ is the symmetric adjacency matrix of the street network of Luxembourg and $\beta=2/\rho(\bm{A})$, where $\rho(\bm{A})$ denotes the spectral radius.
	The chosen probing vectors are random Rademacher and deterministic graph coloring vectors.
	The convergence of the randomized method is significantly improved by investing upfront computations to obtain the graph coloring vectors.}\label{fig:trace_estimation}
\end{figure}

This paper approaches sketching in Krylov subspace methods from an alternative angle.
In the literature on randomized sketching for Krylov methods, the definitions of oblivious subspace embeddings and the choice of sketching matrices are typically tied to one another.
This work, instead, investigates what subspace embeddings are obtained by arbitrary sketching matrices.
This analysis provides an objective for choosing optimal (random or deterministic) sketching matrices over a fixed matrix family: for a given basis of a given subspace, maximize the smallest singular value of its sketched basis.
Exploiting the fact that this objective is well-studied in the context of model order reduction, we propose the use of deterministic row subset selection sketching matrices in sketched Krylov methods.
In analogy to the derandomization of trace estimators, cf.~\Cref{fig:trace_estimation}, this approach requires problem-specific upfront computations.
After this, the sketching matrix is extremely cheap to apply and satisfies subspace embeddings that hold \emph{with probability 1}|in contrast to randomly sketched methods, which satisfy the same \emph{with high probability}.
We propose deterministically sketched Krylov subspace methods for matrix functions, linear systems, and eigenvalue problems that allow the adaptation of existing error bounds of randomly sketched methods to the deterministic case.
Numerical experiments illustrate that they achieve similar accuracies at somewhat longer runtimes to their randomly sketched counterparts.

The term \emph{deterministic sketching} has previously been used in the matrix streaming setting, i.e., where memory constraints only allow accessing the matrix one row at a time \cite{liberty2013simple,ghashami2016frequent}.
Previous works on combining Krylov subspace methods with row or column subset selection include the DEIM\_Arnoldi method representing an inner-product free Arnoldi implementation for eigenvalue problems \cite{du2010implicitly} as well as a CUR Krylov solver for linear matrix equations \cite{akbari2025cur}.
However, this manuscript is, to the best of our knowledge, the first to apply row subset selection techniques as a sketching approach to bases of Krylov subspace methods to solve a broad range of numerical linear algebra problems.

The remainder of this paper is organized as follows.
\Cref{sec:randomized_sketching} introduces Krylov subspace methods, randomized sketching, and the sketched Krylov methods sFOM for matrix functions, sGMRES for linear systems, and sRR for eigenvalue problems.
\Cref{sec:subspace_embeddings} analyzes what subspace embeddings are obtained by arbitrary sketching matrices.
\Cref{sec:deterministic_sketching} introduces the framework of deterministic sketching via row subset selection matrices in Krylov subspace methods and reviews existing row subset selection methods from the model order reduction literature.
\Cref{sec:algorithms} summarizes all ingredients into the deterministically sketched Krylov methods dsFOM for matrix functions, dsGMRES for linear systems, and dsRR for eigenvalue problems.
\Cref{sec:numerics} provides numerical experiments of these methods in comparison to their unsketched and randomly sketched counterparts.

\subsection{Notation}

Throughout this manuscript, $\bm{M}^\ast$ denotes the conjugate transpose and $\bm{M}^\dagger=(\bm{M}^\ast\bm{M})^{-1}\bm{M}^\ast$ the pseudoinverse of a matrix $\bm{M}\in\C^{n\times m}$ with $n\geq m$ and linearly independent columns.
Moreover, the smallest and largest singular values are denoted by $\sigma_{\min}(\bm{M})$ and $\sigma_{\max}(\bm{M})$, respectively, giving rise to the definition of the spectral condition number $\kappa(\bm{M})=\frac{\sigma_{\max}(\bm{M})}{\sigma_{\min}(\bm{M})}$.
The symbols $\bm{0}, \bm{1}\in\C^n$ denote vectors of all zeros and ones, respectively, $\bm{I}\in\C^{n\times n}$ denotes the identity matrix, and $\bm{e}_i\in\C^n$ the $i$-th column of $\bm{I}$.

\section{Randomized sketching in Krylov subspace methods}\label{sec:randomized_sketching}

A wide range of established methods in numerical linear algebra builds on the Krylov subspace.
For a matrix $\bm{A}\in\C^{n\times n}$, a vector $\bm{b}\in\C^n$, and a subspace dimension $m\in\N$, it is defined as
\begin{equation}\label{eq:krylov_subspace}
\mathcal{K}_m(\bm{A},\bm{b}) = \text{span}\{\bm{b},\bm{Ab},\dots,\bm{A}^{m-1}\bm{b}\}.
\end{equation}
Approximations to fundamental linear algebra problems including linear systems, eigenvalue problems, matrix functions, or matrix equations can often be obtained from relatively low-dimensional subspaces in a fast and accurate manner \cite{saad2003iterative,golub2013matrix}.

Since the condition number of the Krylov basis matrix $\begin{bmatrix}\bm{b},\bm{Ab},\dots,\bm{A}^{m-1}\bm{b}\end{bmatrix}$ grows exponentially as a function of $m$ \cite{gautschi1979condition,beckermann2000condition}, the starting point of classical Krylov subspace methods is the generation of an orthonormal basis $\bm{V}_m\in\C^{n\times m}$ of $\mathcal{K}_m(\bm{A},\bm{b})$.
Starting from $\bm{b}/\|\bm{b}\|_2$, each new basis vector is obtained from a matrix-vector product with $\bm{A}$, followed by some form of Gram--Schmidt procedure as well as normalization.
For $\bm{A}$ Hermitian, this is achieved by the Lanczos method \cite{lanczos1950iteration}, which constitutes a three-term recurrence.
Denoting the number of non-zero entries in $\bm{A}$ by $\texttt{nnz}(\bm{A})$, this procedure incurs a computational complexity of $\mathcal{O}(\texttt{nnz}(\bm{A})m+nm)$.
In the non-Hermitian case, full orthogonalization against all previous basis vectors is required in the Arnoldi method \cite{arnoldi1951principle}, which has computational complexity $\mathcal{O}(\texttt{nnz}(\bm{A})m+nm^2)$.
Especially for sparse $\bm{A}$ and large $m$, orthogonalization costs can become the computational bottleneck of the Arnoldi method.

A recent line of work exploits randomized sketching techniques to reduce these orthogonalization costs for non-Hermitian $\bm{A}$ to a linear runtime dependence on $m$.
Sketching refers to embedding vectors from an $m$-dimensional subspace $\V\subset\C^n$ into $\C^s$ with $s\ll n$ by means of a sketching matrix $\bm{S}\in\C^{s\times n}$ that approximately preserves Euclidean lengths \cite{woodruff2014sketching,martinsson2020randomized,epperly2026randomized}.
Such $\bm{S}$ are typically drawn from a family of random matrices with certain statistical guarantees for arbitrary subspaces.

\begin{definition}\label{def:oblivious_subspace_embedding}
	A matrix $\bm{S}\in\C^{s\times n}$ drawn from a distribution of sketching matrices is called an oblivious $\epsilon$-subspace embedding with $\epsilon\in [0,1)$ if
	\begin{equation}\label{eq:epsilon_subspace_embedding}
	(1-\epsilon)\|\bm{v}\|_2^2 \leq \|\bm{Sv}\|_2^2 \leq (1+\epsilon)\|\bm{v}\|_2^2
	\end{equation}
	holds for any $\bm{v}\in\V$ and any subspace $\V\subset\C^n$ \emph{with high probability}.
\end{definition}

Sarl\'os introduced the term Johnson--Lindenstrauss transform for such embedding matrices \cite{sarlos2006improved} in reference to the Johnson--Lindenstrauss Lemma \cite{johnson1984extensions}, which first stated a version of \eqref{eq:epsilon_subspace_embedding} for the embedding of a finite set of points from $\C^n$ to $\C^s$.
Besides other examples such as sparse maps \cite{meng2013low,nelson2013osnap} or matrices with i.i.d.\ Gaussian normal entries \cite{woodruff2014sketching,epperly2026randomized}, a popular choice of sketching matrices satisfying \eqref{eq:epsilon_subspace_embedding} are subsampled random Fourier transforms (SRFT)
\begin{equation}\label{eq:srft}
\bm{S} = \sqrt{\frac{n}{s}} \bm{RHD}.
\end{equation}
Here, $\bm{R}\in\C^{s\times n}$ consists of randomly selected rows of $\bm{I}\in\R^{n\times n}$, $\bm{H}\in\C^{n\times n}$ is a dense but structured orthogonal trigonometric transform, e.g., discrete Fourier, discrete cosine, or Walsh--Hadamard transform, and $\bm{D}\in\C^{n\times n}$ is diagonal with uniformly random diagonal entries $\pm 1$ \cite{sarlos2006improved}.

It should be mentioned that \eqref{eq:epsilon_subspace_embedding} holds with a quantifiable failure probability \cite{sarlos2006improved}.
The required sketching dimension of a Johnson--Lindenstrauss transform to satisfy a fixed failure probability behaves as $s=\mathcal{O}(\epsilon^{-2})$.
For this reason, relatively crude embeddings with $\epsilon=1/\sqrt{2}$ or $\epsilon=1/2$ are typically used in practice, which can be achieved by choosing, e.g., $s=2m$ or $s=4m$, where $m\in\N$ is the dimension of the subspace $\V$ \cite{nakatsukasa2024fast,guttel2023randomized}.

\subsection{Generating the non-orthogonal Krylov basis}\label{sec:krylov_basis}

Oblivious subspace embeddings from \Cref{def:oblivious_subspace_embedding} give rise to the sketch-and-solve paradigm \cite{sarlos2006improved,woodruff2014sketching,martinsson2020randomized,epperly2026randomized}.
Its basic idea is to reduce computations with the original problem matrix $\bm{A}\in\C^{n\times m}$ with $n\geq m$ to $\bm{SA}\in\C^{s\times m}$, where $s\ll n$.
Early uses of this paradigm include, e.g., least-squares problems \cite{sarlos2006improved,rokhlin2008fast} or low-rank approximation \cite{sarlos2006improved,rokhlin2010randomized,halko2011finding}.

More recently, randomized sketching has been introduced for speeding up the generation of the basis $\bm{V}_m\in\C^{n\times m}$ of Krylov subspaces \eqref{eq:krylov_subspace} \cite{balabanov2022randomized,guttel2023randomized,cortinovis2024speeding,nakatsukasa2024fast,burke2024krylov,guttel2024sketch,bucci2025randomized,burke2025gmres,chung2025randomized,de2025arandomized,de2025brandomized,de2025crandomized,guidotti2025accelerating,jang2025randomized,krieger2025general,palitta2025sketched,palitta2025sketchedmateq}.
Instead of the standard orthogonality of $\bm{V}_m$ with respect to the $n$-dimensional Euclidean inner product, several approaches merely impose orthogonality of the sketched basis $\bm{SV}_m\in\C^{s\times m}$ with respect to the $\bm{S}$-inner product $\langle\bm{u},\bm{v}\rangle_{\bm{S}}=\langle\bm{Su},\bm{Sv}\rangle$ for $\bm{u},\bm{v}\in\C^n$.
The randomized Gram--Schmidt process \cite{balabanov2022randomized} achieves this by performing full orthogonalization of every new basis vector with respect to this $\bm{S}$-inner product.
In contrast, the $k$-truncated Arnoldi method that has recently been used in \cite{nakatsukasa2024fast,guttel2023randomized} performs incomplete orthogonalization only against the $k\in\N$ previous basis vectors with respect to the standard $n$-dimensional Euclidean inner product, cf.~\Cref{alg:k_trunc_arnoldi}.
Other techniques such as the Chebyshev recurrence or Newton polynomials may also be used to generate non-orthogonal Krylov bases \cite{nakatsukasa2024fast}.
In these methods, orthogonality in the $\bm{S}$-inner product is prescribed \emph{a-posteriori} by basis whitening \cite{rokhlin2008fast}.

\begin{algorithm}[t]%[htb!]
%	\vspace{0.3em}
	\begin{tabular}{lll}
		Input:
		& $\bm{A}\in\C^{n \times n},$ & Matrix.\\
		& $\bm{b}\in\C^n,$ & Vector.\\
		& $m\in\N,$ & Dimension of $\mathcal{K}_m(\bm{A},\bm{b})$.\\
		& $k\in\N,$ & Orthogonalization truncation\\
		&& parameter.\vspace{.3em}\\
		Output: & $\bm{V}_m = [\bm{v}_1,\dots,\bm{v}_m]\in\C^{n\times m}$, & Non-orthogonal basis of $\mathcal{K}_m(\bm{A},\bm{b})$.\\
		& $\bm{M}_m = [\bm{m}_1,\dots,\bm{m}_m] = \bm{AV}_m,$ & Stored matrix-matrix product.\\
	\end{tabular}
	\vspace{.3em}
	\begin{algorithmic}[1]
		\State $\bm{v}_1 = \bm{b}/\|\bm{b}\|_2$
		\State $\bm{m}_1 = \bm{Av}_1$
		\For{$j=2,\dots,m$}
		\State $\bm{w}_j = (\bm{I} - \bm{v}_{j-1}\bm{v}_{j-1}^\ast - \dots - \bm{v}_{j-k}\bm{v}_{j-k}^\ast)\bm{m}_{j-1}$\Comment{$\bm{v}_i=\bm{0}$ for $i\leq 0$}
		\State $\bm{v}_j = \bm{w}_j/\|\bm{w}_j\|_2$
		\State $\bm{m}_j = \bm{Av}_j$
		\EndFor 
	\end{algorithmic}
	\caption{$k$-truncated Arnoldi method.}\label{alg:k_trunc_arnoldi}
\end{algorithm}

\begin{definition}\label{def:basis_whitening}
	Let $\bm{V}_m\in\C^{n\times m}$ with $n\geq m$ and $\bm{S}\in\C^{s\times n}$ be a subspace embedding with $s\geq m$ and $\text{rank}(\bm{SV}_m)=m$.
	Then the thin QR decomposition \mbox{$\bm{SV}_m=\bm{Q}_m\bm{R}_m$} with $\bm{Q}_m\in\C^{s\times m}$ orthonormal and $\bm{R}_m\in\C^{m\times m}$ upper triangular allows defining the \emph{whitened basis}
	\begin{equation}\label{eq:basis_whitening}
	\widetilde{\bm{V}}_m = \bm{V}_m\bm{R}_m^{-1}.
	\end{equation}
\end{definition}

Clearly, basis whitening orthogonalizes the sketched basis via $\bm{SV}_m\bm{R}_m^{-1}=\bm{Q}_m$ as in the randomized Gram--Schmidt case.
Computationally, performing the $\mathcal{O}(nm^2)$ computation of \eqref{eq:basis_whitening} explicitly should be avoided whenever possible.
Instead, \Cref{sec:sFOM,sec:sGMRES,sec:sRR} present ways to interact with $\bm{R}_m^{-1}$ only by means of $\mathcal{O}(sm^2)$ back-substitution solves.

\begin{remark}\label{rem:krylov_basis_conditioning}
The drawback of the $k$-truncated Arnoldi method is that the condition number of the non-orthogonal basis $\bm{V}_m$ before whitening is often observed to grow towards or beyond inverse machine precision \cite{nakatsukasa2024fast,guttel2023randomized}.
Besides increasing the truncation length $k\in\N$, this issue can be addressed by invoking the sketch-and-select Arnoldi process \cite{guttel2024sketch}, which orthogonalizes against $k$ basis vectors that are chosen from the previously generated basis vectors based on different criteria.
\end{remark}

We close this subsection by noting that the randomized Gram--Schmidt \cite{balabanov2022randomized} and sketch-and-select methods \cite{guttel2024sketch} require access to the full sketching matrix $\bm{S}$ from the first iteration, which makes \emph{oblivious} embeddings suitable candidates.
In contrast, the $k$-truncated Arnoldi method with subsequent basis whitening utilizes the sketching matrix only after the full non-orthogonal basis $\bm{V}_m$ has been generated.
This allows a non-oblivious basis-specific construction of the sketching matrix, which is proposed in \Cref{sec:deterministic_sketching}.
Besides the occurrence of spurious Ritz values outside of the field-of-values of the considered matrix $\bm{A}$, an ill-conditioning of the basis below inverse machine precision has been found to not adversely influence sketched Krylov methods \cite{guttel2023randomized,nakatsukasa2024fast,palitta2025sketched}.

\subsection{Sketched Krylov subspace methods}

This section introduces sketched versions of three established Krylov subspace methods addressing fundamental linear algebra problems: the full orthogonalization method (FOM) for approximating the action of a matrix function on a vector, the generalized minimal residual method \mbox{(GMRES)} for approximately solving a non-Hermitian linear system, and the Rayleigh--Ritz (RR) method for approximating a subset of the eigenvalues and eigenvectors of a matrix.
The starting point of each method is the $k$-truncated Arnoldi method, cf.~\Cref{alg:k_trunc_arnoldi}.
Brief derivations of the sketched methods including one error estimate each are given in the respective \Cref{sec:sFOM,sec:sGMRES,sec:sRR}.

\subsubsection{Matrix functions}\label{sec:sFOM}

A scalar function $f:\C\rightarrow\C$ can be generalized to a mapping $f:\C^{n\times n}\rightarrow\C^{n\times n}$ between square matrices by several equivalent definitions, cf.~\cite[Sec.~1.2]{higham2008functions}.
A common problem in, e.g., differential equations \cite{al2011computing,hochbruck2010exponential,gaudreault2018kiops,bergermann2024adaptive} or network science applications \cite{benzi2020matrix,bergermann2022fast,guttel2023randomized,bergermann2025gradient} is the approximation of the action $f(\bm{A})\bm{b}$ of the matrix function of a given matrix $\bm{A}\in\C^{n\times n}$ on a given vector $\bm{b}\in\C^n$.

We briefly present the sketched FOM (sFOM) method \cite{guttel2023randomized}, for which an equivalent version has been independently proposed in \cite{cortinovis2024speeding}, and refer to \cite{guttel2023randomized} for more details.
A general form of the FOM approximation that admits the use of non-orthogonal Krylov bases $\bm{V}_m$ reads
\begin{equation}\label{eq:FOM}
f(\bm{A})\bm{b} \approx \bm{f}_m^{\mathrm{FOM}} = \bm{V}_m f(\bm{V}_m^\dagger\bm{AV}_m) \bm{V}_m^\dagger\bm{b}.
\end{equation}

\begin{sloppypar}
The starting point of sFOM is using the Cauchy integral definition of matrix functions \cite[Sec.~1.2.3]{higham2008functions}
\begin{equation}
\begin{aligned}\label{eq:fAb_cauchy_integral}
f(\bm{A})\bm{b} = \int_\Gamma f(t) (t\bm{I} - \bm{A})^{-1}\bm{b}d\mu(t) &~\approx \int_\Gamma f(t)\bm{x}_m(t)d\mu(t)\\
&~=  \int_\Gamma f(t)\|\bm{b}\|_2 \bm{V}_m(t\bm{I} - \bm{H}_m)^{-1}\bm{e}_1 d\mu(t),
\end{aligned}
\end{equation}
where $\Gamma$ denotes a closed contour that contains the spectrum of $\bm{A}$ and guarantees the invertibility of $(t\bm{I} - \bm{A})$.
Moreover, $(t\bm{I} - \bm{A})\bm{x}(t)=\bm{b}$ is solved via the FOM approximation for linear systems that includes the Hessenberg reduction $\bm{H}_m\in\C^{m\times m}$ of $\bm{A}$ within the Krylov subspace \cite{saad2003iterative}.
The sketching matrix $\bm{S}\in\C^{s\times n}$ first enters the stage when the typically required orthogonality condition of the residual \mbox{$\bm{r}_m(t) = \bm{b} - (t\bm{I} - \bm{A})\bm{x}_m(t)$} to the Krylov basis $\bm{V}_m$ is loosened to the orthogonality of $\bm{Sr}_m$ to $\bm{SV}_m$.
With $\hat{\bm{x}}_m(t)$ denoting the sketched version of $\bm{x}_m(t)$, the latter condition leads to
\begin{align}
\int_\Gamma f(t)\hat{\bm{x}}_m(t)d\mu(t) &~= \bm{V}_m \int_\Gamma f(t) \left[ (\bm{SV}_m)^\ast (t\bm{SV}_m - \bm{SAV}_m) \right]^{-1} d\mu(t) (\bm{SV}_m)^\ast \bm{Sb}\label{eq:sFOM_inv_integrand}\\
&~= \bm{V}_m (\bm{V}_m^\ast\bm{S}^\ast\bm{SV}_m)^{-1} f\left(\bm{V}_m^\ast\bm{S}^\ast\bm{SAV}_m(\bm{V}_m^\ast\bm{S}^\ast\bm{SV}_m)^{-1}\right) (\bm{SV}_m)^\ast \bm{Sb},\label{eq:sFOM_non_orth}
\end{align}
where the last equality leverages an explicit formula for the inverse in \eqref{eq:sFOM_inv_integrand}.
Since the above derivation is independent of the choice of the basis spanning the Krylov subspace $\mathcal{K}_m(\bm{A},\bm{b})$, basis whitening is employed, cf.~\Cref{def:basis_whitening}.
Replacing $\bm{V}_m$ in \eqref{eq:sFOM_non_orth} by $\widetilde{\bm{V}}_m=\bm{V}_m\bm{R}_m^{-1}$ leads to $\widetilde{\bm{V}}_m^\ast\bm{S}^\ast\bm{S}\widetilde{\bm{V}}_m=\bm{Q}_m^\ast\bm{Q}_m=\bm{I}$ and
\begin{equation}\label{eq:sFOM}
f(\bm{A})\bm{b} \approx \bm{f}_m^{\mathrm{sFOM}} = \bm{V}_m (\bm{R}_m^{-1} f(\bm{Q}_m^\ast \bm{SAV}_m \bm{R}_m^{-1}) \bm{Q}_m^\ast \bm{Sb}).
\end{equation}
The sFOM approximant \eqref{eq:sFOM} can be cheaply evaluated using back-substitution for inverting $\bm{R}_m$ since the argument of the function $f$ is a small $m\times m$ upper Hessenberg matrix \cite{cortinovis2024speeding}.
\end{sloppypar}

A comparison of the FOM \eqref{eq:FOM} and sFOM \eqref{eq:sFOM} approximants for $\bm{SV}_m$ orthonormal yields \cite[Cor.~2.4]{guttel2023randomized}
\begin{equation}\label{eq:sFOM_error}
\|\bm{f}_m^{\mathrm{FOM}} - \bm{f}_m^{\mathrm{sFOM}}\|_2 \leq \sqrt{\frac{1+\epsilon}{1-\epsilon}} \|\bm{b}\|_2 \|f(\bm{V}_m^\dagger\bm{AV}_m) - f(\bm{V}_m^\ast\bm{S}^\ast\bm{SAV}_m)\|_2.
\end{equation}

\subsubsection{Linear systems}\label{sec:sGMRES}

The problem of solving the linear system of equations $\bm{Ax}=\bm{b}$ for a given matrix $\bm{A}\in\C^{n\times n}$ and a given right-hand side $\bm{b}\in\C^n$ for the unknown vector $\bm{x}\in\C^n$ is ubiquitous in linear algebra and its applications, cf.~\cite{saad2003iterative,golub2013matrix} and references therein.

\begin{sloppypar}
Given a starting guess \mbox{$\bm{x}_0\in\C^n$} (such as $\bm{x}_0=\bm{0}$) with initial residual \mbox{$\bm{r}_0=\bm{b}-\bm{Ax}_0$} (such as $\bm{r}_0=\bm{b}$), classical GMRES constructs approximations of the form \mbox{$\bm{x} \approx \bm{x}_0 + \bm{V}_m\widetilde{\bm{y}}$}, where $\bm{V}_m$ denotes the basis of the Krylov subspace $\mathcal{K}_m(\bm{A},\bm{r}_0)$ \cite{saad1986gmres,saad2003iterative}.
The vector $\widetilde{\bm{y}}\in\C^m$ taking linear combinations of the columns of $\bm{V}_m$ is chosen as to minimize the residual $\bm{r}=\bm{b}-\bm{Ax} = \bm{r}_0-\bm{AV}_m\bm{y}$, i.e.,
\begin{equation}\label{eq:GMRES_LS}
\widetilde{\bm{y}} = \arg\min_{\bm{y}\in\C^m} \|\bm{AV}_m\bm{y} - \bm{r}_0\|_2 = (\bm{AV}_m)^\dagger\bm{r}_0.
\end{equation}
The GMRES approximation is then given by
\begin{equation}\label{eq:GMRES}
\bm{x}\approx\bm{x}_m^{\mathrm{GMRES}} = \bm{x}_0 + \bm{V}_m\widetilde{\bm{y}} = \bm{x}_0 + \bm{V}_m(\bm{AV}_m)^\dagger\bm{r}_0.
\end{equation}
\end{sloppypar}

\begin{sloppypar}
The sketched version sGMRES \cite{nakatsukasa2024fast} builds around sketching the overdetermined least-squares problem \eqref{eq:GMRES_LS} \cite{rokhlin2008fast}
\begin{equation}\label{eq:sGMRES_LS}
\widehat{\bm{y}} = \arg\min_{\bm{y}\in\C^m} \|\bm{S}(\bm{AV}_m\bm{y} - \bm{r}_0)\|_2 = (\bm{SAV}_m)^\dagger(\bm{Sr}_0).
\end{equation}
Then, assuming $\text{rank}(\bm{SAV}_m)=m$, the thin QR-decomposition $\bm{S}\bm{AV}_m = \bm{Q}_m\bm{R}_m$ is computed\footnote{Note that this step differs from \Cref{sec:sFOM,sec:sRR}, where basis whitening, i.e., the QR-decomposition of $\bm{S}\bm{V}_m$ is computed.}.
With this, the sGMRES approximant becomes
\begin{equation}\label{eq:sGMRES}
\bm{x}\approx\bm{x}_m^{\mathrm{sGMRES}} =  \bm{x}_0 + \bm{V}_m\widehat{\bm{y}} = \bm{x}_0 + \bm{V}_m(\bm{R}_m^{-1} \bm{Q}_m^\ast \bm{Sr}_0).
\end{equation}
Basic ingredients such as preconditioning or restarting that aim at reducing the required number of Krylov iterations or basis vectors to store can be incorporated into sGMRES in a straightforward manner \cite[Sec.~3.5--3.6]{nakatsukasa2024fast}.
\end{sloppypar}

Assuming that $\bm{S}$ embeds $\text{Ran}([\bm{r}_0,\bm{AV}_m])$, the relation \cite[Eq.~(3.5)]{nakatsukasa2024fast}
\begin{equation}\label{eq:sGMRES_error}
\|\bm{Ax}_m^{\mathrm{GMRES}} - \bm{b}\|_2 \leq \|\bm{Ax}_m^{\mathrm{sGMRES}} - \bm{b}\|_2 \leq \sqrt{\frac{1+\epsilon}{1-\epsilon}} \|\bm{Ax}_m^{\mathrm{GMRES}} - \bm{b}\|_2
\end{equation}
between the GMRES and sGMRES residuals is directly inherited from standard analysis of the sketched least-squares problem \eqref{eq:sGMRES_LS}.

\subsubsection{Eigenvalue problems}\label{sec:sRR}

The final linear algebra problem considered in this section is the problem of finding a subset of the eigenvalues $\lambda_i\in\C$ and corresponding eigenvectors $\bm{x}_i\in\C^n, i=1,\dots,m$ for a given matrix $\bm{A}\in\C^{n\times n}$ \cite{saad2011numerical,golub2013matrix}.

In \cite{nakatsukasa2024fast}, it is argued that the most natural formulation of the Rayleigh--Ritz method is to seek a non-zero vector $\bm{x}=\bm{V}_m\bm{y}\in\mathcal{K}_m(\bm{A},\bm{b})$ for some $\bm{b}\in\C^n$ such that the eigenvector residual $\bm{r}=\bm{Ax}-\lambda\bm{x}$ is orthogonal to the basis $\bm{V}_m$ of $\mathcal{K}_m(\bm{A},\bm{b})$.
Formally, this leads to the problem of finding $\bm{y}\in\C^m$ and $\lambda\in\C$ such that
$$
\bm{V}_m^\ast(\bm{AV}_m\bm{y} - \lambda\bm{V}_m\bm{y}) = \bm{0} \Leftrightarrow (\bm{V}_m^\dagger\bm{AV}_m)\bm{y} = \lambda\bm{y},
$$
which corresponds to a small $m\times m$ eigenvalue problem in the matrix $\widetilde{\bm{M}}= \bm{V}_m^\dagger\bm{AV}_m$.
In the classical case where $\bm{V}_m$ is the orthonormal basis generated by the Arnoldi method, $\widetilde{\bm{M}}$ is the corresponding upper Hessenberg matrix.
Eigenpairs $(\widetilde{\bm{y}},\widetilde{\lambda})$ of $\widetilde{\bm{M}}$ are called Ritz-vectors and -values and the tuples $(\bm{x}^{\mathrm{RR}},\lambda^{\mathrm{RR}}) = (\bm{V}_m\widetilde{\bm{y}},\widetilde{\lambda})$ form approximate eigenpairs of $\bm{A}$ \cite{golub2013matrix}.

An alternative formulation considers the rectangular eigenvalue problem of minimizing the eigenvector residual \emph{over} the Krylov subspace $\mathcal{K}_m(\bm{A},\bm{b})$ \cite{nakatsukasa2024fast}, i.e.,
\begin{equation}\label{eq:rectangular_eval_prob}
\min_{\bm{y}\in\C^m,\lambda\in\C} \|\bm{AV}_m\bm{y} - \lambda\bm{V}_m\bm{y}\|_2 \quad \text{s.t.}\quad \|\bm{V}_m\bm{y}\|_2=1.
\end{equation}
While the Rayleigh--Ritz method does not solve the problem \eqref{eq:rectangular_eval_prob}, any eigenpair $(\widetilde{\bm{y}},\widetilde{\lambda})$ of $\widetilde{\bm{M}} = \bm{V}_m^\dagger\bm{AV}_m$ satisfies
$$
\|\bm{AV}_m\widetilde{\bm{y}} - \widetilde{\lambda}\bm{V}_m\widetilde{\bm{y}}\|_2 = \|(\bm{AV}_m - \bm{V}_m\widetilde{\bm{M}})\widetilde{\bm{y}}\|_2
$$
and Rayleigh--Ritz does solve the related variational eigenvalue problem \cite{nakatsukasa2024fast}
\begin{equation}\label{eq:sRR_variational_problem}
\min_{\bm{M}\in\C^{m\times m}} \| \bm{AV}_m - \bm{V}_m\bm{M} \|_F.
\end{equation}

Similarly to sGMRES, the sketched Rayleigh--Ritz (sRR) method \cite{nakatsukasa2024fast} builds around formulating the classical Rayleigh--Ritz method in terms of a sketched least-squares problem.
Sketching \eqref{eq:sRR_variational_problem} leads to
$$
\min_{\bm{M}\in\C^{m\times m}} \| \bm{S}(\bm{AV}_m - \bm{V}_m\bm{M}) \|_F,
$$
which has the solution $\widehat{\bm{M}} = (\bm{SV}_m)^ \dagger(\bm{SAV}_m) = \bm{R}_m^{-1}(\bm{Q}_m^\ast(\bm{S}\bm{AV}_m))$, where basis whitening is applied in the last equality, cf.~\Cref{def:basis_whitening}.
Solving the small $m\times m$ eigenvalue problem
\begin{equation}\label{eq:sRR}
(\bm{R}_m^{-1}\bm{Q}_m^\ast (\bm{SAV}_m)) \bm{y}_i = \lambda_i\bm{y}_i
\end{equation}
yields the approximate eigenvalues $\lambda_1^{\mathrm{sRR}},\dots,\lambda_m^{\mathrm{sRR}}$ of $\bm{A}$, while the corresponding approximate eigenvectors are obtained via $
\bm{x}_i^{\mathrm{sRR}} = \frac{\bm{V}_m\bm{y}_i}{\|\bm{V}_m\bm{y}_i\|}, i=1,\dots,m$.

For sRR, the following relation holds between the sketched and unsketched eigenvector residuals \cite[Eq.~(6.11)]{nakatsukasa2024fast}
\begin{equation}\label{eq:sRR_error}
\begin{aligned}
\sqrt{\frac{1-\epsilon}{1+\epsilon}} \frac{\|\bm{S}(\bm{AV}_m\bm{y}_i - \lambda_i\bm{V}_m\bm{y}_i)\|_2}{\|\bm{SV}_m\bm{y}_i\|_2} &~\leq \frac{\|\bm{AV}_m\bm{y}_i - \lambda_i\bm{V}_m\bm{y}_i\|_2}{\|\bm{V}_m\bm{y}_i\|_2}\\
&~\leq \sqrt{\frac{1+\epsilon}{1-\epsilon}} \frac{\|\bm{S}(\bm{AV}_m\bm{y}_i - \lambda_i\bm{V}_m\bm{y}_i)\|_2}{\|\bm{SV}_m\bm{y}_i\|_2},
\end{aligned}
\end{equation}
for all $i=1,\dots,m,$ assuming that $\bm{S}$ embeds $\text{Ran}([\bm{V}_m,\bm{AV}_m])$.

\section{General subspace embeddings}\label{sec:subspace_embeddings}

The exposition in \Cref{sec:randomized_sketching} aims to reflect the angle from which sketching is typically introduced in the numerical linear algebra context.
It suggests a certain historically grounded intertwinedness of the subspace embedding property \eqref{eq:epsilon_subspace_embedding} and the choice of randomized sketching transform.
In particular, sparse maps and SRFTs appear to have established themselves as unchallenged sketching transform of choice.
However, questions about more general sketching matrices and criteria for their effectiveness have recently been raised \cite{amsel2026linear}.

This section offers a different angle on sketching in numerical linear algebra by analyzing what subspace embeddings are obtained by \emph{arbitrary} sketching matrices.

\begin{proposition}\label{prop:deterministic_subspace_embedding}
	Let $\bm{S}\in\C^{s\times n}$ be \emph{any} matrix with $s\geq m$ and $\bm{V}\in\C^{n\times m}$ a (generally non-orthogonal) basis of an $m$-dimensional subspace $\V\subset\C^n$.
	Then we have
	\begin{equation}\label{eq:general_subspace_embedding}
	\frac{\sigma_{\min}^2(\bm{SV})}{\sigma_{\max}^2(\bm{V})}\|\bm{v}\|_2^2 \leq \|\bm{Sv}\|_2^2 \leq \frac{\sigma_{\max}^2(\bm{SV})}{\sigma_{\min}^2(\bm{V})} \|\bm{v}\|_2^2,
	\end{equation}
	for any $\bm{v}\in\mathcal{V}$.
\end{proposition}
\begin{proof}
	Since $\V=\text{Ran}(\bm{V})$, we may write every element $\bm{v}\in\V$ as $\bm{v}=\bm{V}\bm{y}$ for some $\bm{y}\in\C^m$.
	Throughout the proof, we assume $\bm{v}\neq\bm{0}$ and $\bm{y}\neq\bm{0}$, for which \eqref{eq:general_subspace_embedding} holds with equality.
	By
	\begin{equation}\label{eq:2norm_subordinarity}
	\|\bm{V}\bm{y}\|_2^2 \leq \|\bm{V}\|_2^2 \|\bm{y}\|_2^2 = \sigma_{\max}^2(\bm{V}) \|\bm{y}\|_2^2,
	\end{equation}
	we have
	\begin{equation*}
	\frac{\|\bm{SV}\bm{y}\|_2^2}{\|\bm{V}\bm{y}\|_2^2} \geq \frac{\|\bm{SV}\bm{y}\|_2^2}{\sigma_{\max}^2(\bm{V}) \|\bm{y}\|_2^2} = \frac{1}{\sigma_{\max}^2(\bm{V})} \frac{\bm{y}^\ast\bm{V}^\ast\bm{S}^\ast\bm{SV}\bm{y}}{\bm{y}^\ast\bm{y}} \geq \frac{\sigma_{\min}^2(\bm{SV})}{\sigma_{\max}^2(\bm{V})},
	\end{equation*}
	where the last inequality uses the fact that the Rayleigh quotient of the matrix $(\bm{SV})^\ast(\bm{SV})$ is minimized by its eigenvector to the smallest eigenvalue, or equivalently, the right singular vector to the square of the smallest singular value of $\bm{SV}$.
	
	For the upper bound, we again use \eqref{eq:2norm_subordinarity} to obtain
	\begin{equation*}
	\frac{\|\bm{SV}\bm{y}\|_2^2}{\|\bm{V}\bm{y}\|_2^2} \leq \frac{\|\bm{SV}\|_2^2 \|\bm{y}\|_2^2}{\|\bm{V}\bm{y}\|_2^2} = \sigma_{\max}^2(\bm{SV}) \frac{\|\bm{y}\|_2^2}{\|\bm{V}\bm{y}\|_2^2}.
	\end{equation*}
	Moreover, we have
	\begin{equation*}
	\sigma_{\min}^2(\bm{V}) \leq \frac{\bm{y}^\ast\bm{V}^\ast\bm{V}\bm{y}}{\bm{y}^\ast\bm{y}} \Leftrightarrow  \frac{\bm{y}^\ast\bm{y}}{\bm{y}^\ast\bm{V}^\ast\bm{V}\bm{y}} \leq \frac{1}{\sigma_{\min}^2(\bm{V})}
	\end{equation*}
	and hence
	\begin{equation*}
	\frac{\|\bm{SV}\bm{y}\|_2^2}{\|\bm{V}\bm{y}\|_2^2} \leq \frac{\sigma_{\max}^2(\bm{SV})}{\sigma_{\min}^2(\bm{V})}.
	\end{equation*}
\end{proof}

With \Cref{prop:deterministic_subspace_embedding}, the distortion factor caused by the embedding with an arbitrary sketching matrix $\bm{S}\in\C^{s\times n}$ is upper-bounded by
\begin{equation}\label{eq:distortion_cond_numbers}
\sqrt{\frac{\frac{\sigma_{\max}^2(\bm{SV})}{\sigma_{\min}^2(\bm{V})}}{\frac{\sigma_{\min}^2(\bm{SV})}{\sigma_{\max}^2(\bm{V})}}}  = \frac{\sigma_{\max}(\bm{SV}) \sigma_{\max}(\bm{V})}{\sigma_{\min}(\bm{SV})\sigma_{\min}(\bm{V})} = \kappa(\bm{SV}) \kappa(\bm{V}).
\end{equation}
For $\bm{V}$ orthonormal and $\bm{S}$ injective on $\mathcal{V}$, \eqref{eq:distortion_cond_numbers} represents the exact distortion factor.

\subsection{Whitening the non-orthogonal Krylov basis}\label{sec:whitened_basis}

Viewing \eqref{eq:distortion_cond_numbers} in the case of non-orthogonal Krylov bases $\bm{V}_m$ obtained from the $k$-truncated Arnoldi method, cf.~\Cref{alg:k_trunc_arnoldi}, appears discouraging.
As discussed in \Cref{rem:krylov_basis_conditioning}, $\kappa(\bm{V}_m)$ may be impractically large.

Better news can be expected after \emph{whitening} the basis, cf.~\Cref{def:basis_whitening}, i.e., changing the basis of $\mathcal{V}$ while leaving the action of $\bm{S}$ on $\mathcal{V}$ unchanged.
The following is a direct consequence of \Cref{prop:deterministic_subspace_embedding}.
\begin{corollary}\label{cor:whitened_embedding}
	Let $\bm{V}_m\in\C^{n\times m}$ with $n\geq m$, $\bm{S}\in\C^{s\times n}$ be an arbitrary embedding matrix with $s\geq m$, $\text{rank}(\bm{SV}_m)=m$, and $\bm{SV}_m=\bm{Q}_m\bm{R}_m$ a thin QR-decomposition with the whitened basis $\widetilde{\bm{V}}_m=\bm{V}_m\bm{R}_m^{-1}$.
	Then, the embedding of the subspace $\mathcal{V}\subset\C^n$ spanned by $\bm{V}_m$ and $\widetilde{\bm{V}}_m$ after basis whitening satisfies
	\begin{equation}\label{eq:whitened_subspace_embedding}
	\frac{1}{\sigma_{\max}^2(\bm{V}_m\bm{R}_m^{-1})} \|\bm{v}\|_2^2 \leq \|\bm{Sv}\|_2^2 \leq \frac{1}{\sigma_{\min}^2(\bm{V}_m\bm{R}_m^{-1})} \|\bm{v}\|_2^2,
	\end{equation}
	for any $\bm{v}\in\mathcal{V}$.
	The embedding \eqref{eq:whitened_subspace_embedding} is sharp since for $\bm{v}=\widetilde{\bm{V}}_m\bm{y}, \bm{y}\neq\bm{0}$, we have $\frac{\|\bm{Sv}\|_2}{\|\bm{v}\|_2}=\frac{\|\bm{y}\|_2}{\|\widetilde{\bm{V}}_m\bm{y}\|_2}$.
	Moreover, the exact subspace distortion factor reads
	\begin{equation}\label{eq:cond_V_R} \kappa(\bm{V}_m\bm{R}_m^{-1}) = \frac{\sigma_{\max}(\bm{V}_m\bm{R}_m^{-1})}{\sigma_{\min}(\bm{V}_m\bm{R}_m^{-1})}.
	\end{equation}
\end{corollary}
Note that in contrast to \eqref{eq:epsilon_subspace_embedding}, the relations \eqref{eq:general_subspace_embedding}, \eqref{eq:whitened_subspace_embedding}, and \eqref{eq:cond_V_R} hold with probability $1$ for a fixed matrix $\bm{S}$.

\begin{remark}\label{rem:JL_distortion}
\Cref{cor:whitened_embedding} can be taken as starting point for evaluating the suitability of arbitrary matrices $\bm{S}\in\C^{s\times n}$ as sketching matrix for problems involving the basis $\bm{V}_m$.
Interestingly, the relation
\begin{equation*}
\kappa(\bm{V}_m\bm{R}_m^{-1}) \leq \sqrt{\frac{1+\epsilon}{1-\epsilon}}
\end{equation*}
that holds with high probability for any realization of a Johnson--Lindenstrauss transform, cf., e.g., \cite[Cor.~2.2]{balabanov2022randomized} and \cite[Prop.~2.1]{palitta2025sketched}, follows from \Cref{cor:whitened_embedding} as the special case in which \mbox{$\frac{1}{\sigma_{\max}^2(\bm{V}_m\bm{R}_m^{-1})}\geq1-\epsilon$} and $\frac{1}{\sigma_{\min}^2(\bm{V}_m\bm{R}_m^{-1})}\leq1+\epsilon$ hold with high probability, cf.~\eqref{eq:epsilon_subspace_embedding} and \eqref{eq:whitened_subspace_embedding}.
%However, using one fixed realization of a Johnson--Lindenstrauss transform (e.g., by fixing the random seed) permits the evaluation of \eqref{eq:cond_V_R} to obtain the corresponding subspace distortion factor that holds with probability 1.
\end{remark}

More generally, one may now choose some family of (randomized or deterministic) sketching matrices and identify an optimal member minimizing the induced subspace distortion \eqref{eq:cond_V_R}.
Assuming the singular values of $\bm{V}_m$ to be fixed, the goal becomes to choose $\bm{S}$ such that the \emph{sketched} basis $\bm{SV}_m$ matches the \emph{original} basis $\bm{V}_m$ as closely as possible in a spectral sense\footnote{Note that by the thin QR-decomposition $\bm{SV}_m=\bm{Q}_m\bm{R}_m$, the singular values of $\bm{SV}_m$ and $\bm{R}_m$ coincide.}.

To this end, we separately consider numerator and denominator of \eqref{eq:cond_V_R}.
We have
$$
\sigma_{\max}(\bm{V}_m\bm{R}_m^{-1}) \leq  \sigma_{\max}(\bm{V}_m)\sigma_{\max}(\bm{R}_m^{-1}) = \frac{\sigma_{\max}(\bm{V}_m)}{\sigma_{\min}(\bm{R}_m)} = \frac{\sigma_{\max}(\bm{V}_m)}{\sigma_{\min}(\bm{SV}_m)},
$$
which can be minimized by maximizing $\sigma_{\min}(\bm{SV}_m)$ with respect to $\bm{S}$.
Moreover, since we have $\bm{V}_m\in\C^{n\times m}$ with $n\geq m$, it holds that
$$
\sigma_{\min}(\bm{V}_m\bm{R}_m^{-1}) \geq \sigma_{\min}(\bm{V}_m)\sigma_{\min}(\bm{R}_m^{-1}) = \frac{\sigma_{\min}(\bm{V}_m)}{\sigma_{\max}(\bm{R}_m)} = \frac{\sigma_{\min}(\bm{V}_m)}{\sigma_{\max}(\bm{SV}_m)},
$$
which can be maximized by minimizing $\sigma_{\max}(\bm{SV}_m)$ with respect to $\bm{S}$.
We now turn to deciding, which of the two conditions has a bigger impact on the magnitude of \eqref{eq:cond_V_R}.

\begin{lemma}\label{lem:sigma_max_V}
	Let $\bm{V}_m\in\C^{n\times m}$ with $n\geq m$ be a (generally non-orthogonal) matrix with normalized columns.
	Then, we have
	$$
	\sigma_{\max}(\bm{V}_m)\leq \sqrt{m}.
	$$
\end{lemma}
\begin{proof}
	\begin{sloppypar}
	Since the columns of $\bm{V}_m$ are normalized in $2$-norm, all entries of \mbox{$\bm{V}_m^\ast\bm{V}_m\in\C^{m\times m}$} have absolute value less than or equal to $1$.
	Then,
	\begin{equation*}
	\sigma_{\max}^2(\bm{V}_m) = \|\bm{V}_m^\ast\bm{V}_m\|_2 \leq \|\bm{V}_m^\ast\bm{V}_m\|_F \leq m. \qed
	\end{equation*}
\end{sloppypar}
\end{proof}

\begin{remark}\label{rem:sigma_min_V}
\begin{sloppypar}
The smallest singular value $\sigma_{\min}(\bm{V}_m)$ for a full-rank matrix \mbox{$\bm{V}_m\in\C^{n\times m}$} with $n\geq m$ may be arbitrarily small in the case of near-colinearity of its columns.
Depending on the distortion of Euclidean norms induced by the sketching matrix $\bm{S}$, this observation as well as \Cref{lem:sigma_max_V} carry over to $\sigma_{\min}(\bm{SV}_m)$ and $\sigma_{\max}(\bm{SV}_m)$ modulo the subspace distortion factor.
Although it does not guarantee the global minimization of the distortion \eqref{eq:cond_V_R}, the maximization of $\sigma_{\min}(\bm{SV}_m)$ with respect to $\bm{S}$ hence represents a suitable objective to approximately minimize \eqref{eq:cond_V_R}.
\end{sloppypar}
\end{remark}

In summary, this section derives a criterion for evaluating the efficacy of arbitrary matrices $\bm{S}\in\C^{n\times s}$ as sketching matrices.
In particular, sketching matrices for non-orthogonal Krylov bases $\bm{V}_m$ should be designed such that the sketched basis $\bm{SV}_m$ is close to $\bm{V}_m$ in a spectral sense to yield small subspace distortion factors \eqref{eq:cond_V_R}.
Since the largest singular values of $\bm{V}_m$ and $\bm{SV}_m$ are moderate, cf.~\Cref{lem:sigma_max_V,rem:sigma_min_V}, we conclude with the objective
\begin{equation}\label{eq:objective_max_sigma_min}
\max_{\bm{S}} \sigma_{\min}(\bm{SV}_m),
\end{equation}
to be optimized over a fixed family of sketching matrices, subject to a suitable normalization or boundedness condition.
The importance of minimizing subspace distortion factors \eqref{eq:cond_V_R} can be seen at the example of the error bounds \eqref{eq:sFOM_error}, \eqref{eq:sGMRES_error}, and \eqref{eq:sRR_error} for Krylov methods sketched with Johnson--Lindenstrauss transforms.
Here, the distortion factor $\kappa(\bm{V}_m\bm{R}_m^{-1}) \leq \sqrt{\frac{1+\epsilon}{1-\epsilon}}$ enters as a factor by which the output of a sketched method deviates from the output of its unsketched counterpart.

\section{Deterministic sketching via row subset selection}\label{sec:deterministic_sketching}

This section proposes an alternative approach to sketching in numerical linear algebra.
Instead of drawing sketching matrices from random matrix distributions that satisfy subspace embeddings \emph{with high probability}, we propose the use of \emph{deterministic} sketching matrices satisfying the same property \emph{with probability 1}.

In the spirit of derandomizing stochastic trace estimators, cf.~\Cref{fig:trace_estimation} and its description in \Cref{sec:intro}, we propose to invest upfront computations to generate problem-specific sketching matrices, which are then extremely cheap to apply in sketched Krylov methods.

Devising such deterministic sketching matrices builds upon the analysis from \Cref{sec:subspace_embeddings} that identifies \eqref{eq:objective_max_sigma_min} as an objective to generate suitable sketching matrices for a given non-orthogonal Krylov basis $\bm{V}_m$.

Among an infinity of options, we choose the family of row subset selection matrices.
This is motivated by the fact that \emph{randomized} row subset selection is used both as the first component of SRFT's \eqref{eq:srft} and as standalone sketching approach \cite[Sec.~9.6]{martinsson2020randomized} as well as analogies between \Cref{sec:subspace_embeddings} and model order reduction, cf.~\Cref{sec:deim_qdeim,sec:over-sampling}.

\begin{definition}\label{def:RSS_sketching_matrix}
	For a vector of $s\in\N$ non-repeated indices $\bm{p}\in\{1,\dots,n\}^s$, the corresponding deterministic row subset selection sketching matrix is defined as
	\begin{equation}
	\bm{S}=\bm{I}(\bm{p},:)=[\bm{e}_{\bm{p}_1},\dots,\bm{e}_{\bm{p}_s}]^\ast\in\{0,1\}^{s\times n}.
	\end{equation}
\end{definition}

Once the index vector $\bm{p}$ has been determined by upfront computations, applying $\bm{S}$ for extracting $s$ rows from a vector is an extremely fast $\mathcal{O}(s)$ operation.
Although it is not the case in the methods considered in this manuscript, this approach would be particularly efficient for methods that require many applications of the sketching matrix.

In contrast to \Cref{def:oblivious_subspace_embedding}, deterministic row subset selection sketching matrices give rise to somewhat different subspace embeddings.

\begin{lemma}
	Let $\V\subset\C^n$ and $\bm{S}\in\{0,1\}^{s\times n}$ be a deterministic row subset selection sketching matrix as defined in \Cref{def:RSS_sketching_matrix} that is injective on $\V$.
	Then, there exist fixed constants $\epsilon,\delta\in[0,1)$ with $\epsilon\geq\delta$ such that
	\begin{equation}\label{eq:det_epsilon_delta_embedding}
	(1-\epsilon)\|\bm{v}\|_2^2 \leq \|\bm{Sv}\|_2^2 \leq (1-\delta) \|\bm{v}\|_2^2,
	\end{equation}
	holds with probability 1 for any $\bm{v}\in\V$.
	The constants $\epsilon$ and $\delta$\footnote{not to be confused with the failure probability of oblivious subspace embeddings} depend on $\bm{S}$ as well as the basis of $\V$.
\end{lemma}
\begin{proof}
	We clearly have $0\leq\frac{\|\bm{Sv}\|_2^2}{\|\bm{v}\|_2^2}\leq 1$ for $\bm{v}\neq\bm{0}$ since $\bm{S}$ extracts a subset of the entries of $\bm{v}$.
	The injectivity of $\bm{S}$ on $\V$ yields $\|\bm{Sv}\|_2>0$.
	For $\bm{v}=\bm{0}$, \eqref{eq:det_epsilon_delta_embedding} holds with equality.
	The dependence of $\epsilon$ and $\delta$ on $\bm{S}$ and the basis $\bm{V}_m$ of $\V$ follows from \Cref{prop:deterministic_subspace_embedding} and \Cref{cor:whitened_embedding}.
\end{proof}

\begin{sloppypar}
An unusual property of \eqref{eq:det_epsilon_delta_embedding} is the systematic Euclidean length reduction \mbox{$\|\bm{Sv}\|_2^2\leq\|\bm{v}\|_2^2$}, while most random embeddings are isotropic and guarantee length preservation in expectation, i.e., $\mathbb{E}\|\bm{Sv}\|_2^2 = \|\bm{v}\|_2^2$ \cite{epperly2026randomized}.
A similar property may be restored for \eqref{eq:det_epsilon_delta_embedding} by scaling the matrix $\bm{S}$ by a factor $c\in\R$ that, e.g., guarantees norm preservation in expectation over random linear combinations of basis elements of $\V$ drawn from a suitable probability distribution.
Such factor would, however, be inconsequential for the corresponding subspace distortion \eqref{eq:cond_V_R} as well as the sketched approximations \eqref{eq:sFOM}, \eqref{eq:sGMRES}, and \eqref{eq:sRR} since basis whitening\footnote{and in the case of sGMRES, the QR decomposition $\bm{SAV}_m=\bm{Q}_m\bm{R}_m$} would merely introduce the factor $\frac{1}{c}$ in $\bm{R}_m^{-1}$, cf.~\Cref{def:basis_whitening}.
\end{sloppypar}

\subsection{DEIM and Q-DEIM}\label{sec:deim_qdeim}

This subsection introduces two strategies for choosing an index vector $\bm{p}$ from \Cref{def:RSS_sketching_matrix} of length $s=m$.
Both have been proposed in the context of model order reduction, which is briefly described in \Cref{sec:MOR}.
The objective of the two methods is the maximization of $\sigma_{\min}(\bm{SV}_m)$, cf.~\Cref{sec:subspace_embeddings}, over the family of row subset selection matrices $\bm{S}\in\C^{s\times n}$ for a given basis $\bm{V}_m\in\C^{n\times m}$ of an $m$-dimensional subspace of $\C^n$.

As first method, the discrete empirical interpolation method (DEIM) summarized in \Cref{alg:deim} iteratively processes the columns of the basis $\bm{V}_m$.
In each iteration $j$, one row index $\bm{p}_j$ is selected via the largest magnitude entry (line 6) of a residual between the current column $\bm{v}_j$ and its current approximation $\bm{Vc}$ (line 5).
Eventually, DEIM outputs the index vector $\bm{p}\in\{1,\dots,n\}^m$ where the corresponding row subset selection matrix $\bm{S}\in\{0,1\}^{s\times n}$ extracts $s=m$ linearly independent rows from $\bm{V}_m$, guaranteeing $\sigma_{\mathrm{min}}(\bm{SV}_m)>0$.

\begin{algorithm}[t]%[htb!]
%	\vspace{0.3em}
	\begin{tabular}{lll}
		Input:
		& $\bm{V}_m=[\bm{v}_1,\dots,\bm{v}_m]\in\C^{n \times m},$ & Full-rank basis.\vspace{.3em}\\
		Output: & $\bm{p}\in\{1,\dots,n\}^m$, & Vector of non-repeated row indices.
	\end{tabular}
	\vspace{.3em}
	\begin{algorithmic}[1]
		\State $\bm{p}_1 = \arg\max |\bm{v}_1|$
		\State $\bm{V} = \bm{v}_1, \bm{S} = \bm{e}_{\bm{p}_1}^\top, \bm{p} = \bm{p}_1$
		\For{$j = 2,\dots,m$}
		\State Solve $(\bm{SV})\bm{c} = \bm{Sv}_j$ for $\bm{c}$
		\State $\bm{r} = \bm{v}_j - \bm{Vc}$
		\State $\bm{p}_j = \arg\max |\bm{r}|$
		\State $\bm{V}\leftarrow[\bm{V},\bm{v}_j], \bm{S}\leftarrow\begin{bmatrix}
		\bm{S}\\\bm{e}_{\bm{p}_j}^\top
		\end{bmatrix},
		\bm{p}\leftarrow\begin{bmatrix}
		\bm{p}\\\bm{p}_j
		\end{bmatrix}$
		\EndFor
	\end{algorithmic}
	\caption{Discrete Empirical Interpolation Method (DEIM).}\label{alg:deim}
\end{algorithm}

In the model order reduction context, DEIM aims at approximating vectors $\bm{f}\in\C^n$ in the subspace spanned by an orthonormal basis $\bm{V}_m$ via $\bm{f} = \bm{V}_m\bm{c}$ with coefficient vector $\bm{c} = (\bm{S}\bm{V}_m)^{-1}\bm{S}\bm{f}$, cf.~\Cref{sec:MOR}.
Its goal is the minimization of the approximation error \cite[Lem.~3.2]{chaturantabut2010nonlinear}
\begin{equation}\label{eq:deim_error}
\|\bm{f} - \bm{V}_m(\bm{S}\bm{V}_m)^{-1}\bm{Sf}\|_2 \leq \|(\bm{S}\bm{V}_m)^{-1}\|_2 \|(\bm{I}-\bm{V}_m\bm{V}_m^\ast)\bm{f}\|_2,
\end{equation}
which deviates from the best approximation in the subspace spanned by $\bm{V}_m$ by the factor $\|(\bm{S}\bm{V}_m)^{-1}\|_2 = \frac{1}{\sigma_{\min}(\bm{S}\bm{V}_m)}$.
DEIM minimizes this factor locally in each iteration \cite{chaturantabut2010nonlinear}, or equivalently, maximizes $\sigma_{\min}(\bm{S}\bm{V}_m)$, making it an ideal candidate to address the objective \eqref{eq:objective_max_sigma_min} in the case of orthonormal $\bm{V}_m$.

Although the theory is less developed for non-orthogonal bases $\bm{V}_m$, we propose to apply DEIM to select index vectors $\bm{p}$ in the setting of this work.
Algorithmically, orthogonality is not a requirement as DEIM is equivalent to partial left-looking pivoting without replacement in computing LU decompositions \cite{sorensen2016deim,drmac2016new}, which applies to general matrices.
Moreover, it has been noted that once the DEIM interpolation indices are picked, say on an orthonormal basis of the subspace $\V$ spanned by the non-orthogonal $\bm{V}_m$, \eqref{eq:deim_error} holds independently of the chosen basis \cite{chaturantabut2010nonlinear}.
\emph{A-priori} bounds on $\|(\bm{S}\bm{V}_m)^{-1}\|_2$ are available for orthonormal $\bm{V}_m$ \cite{chaturantabut2010nonlinear}, however, these are typically too loose to be of practical use and we do not make any use of such results in this work.

The described equivalence between DEIM and pivoting in matrix factorizations led to a ``surprisingly simple and effective'' \cite[Sec.~1.4]{drmac2016new} variant of DEIM, designated Q-DEIM.
As a second method, it chooses the vector of row indices $\bm{p}$ as the first $m$ pivots of a pivoted QR decomposition of $\bm{V}_m^\ast$, cf.~\Cref{alg:q-deim}.
For orthonormal bases, Q-DEIM yields improved \emph{a-priori} bounds on $\|(\bm{S}\bm{V}_m)^{-1}\|_2$ and is independent under orthogonal transformations of the basis $\bm{V}_m$ \cite{drmac2016new}.

\begin{algorithm}[t]%[htb!]
%	\vspace{0.3em}
	\begin{tabular}{lll}
		Input:
		& $\bm{V}_m=[\bm{v}_1,\dots,\bm{v}_m]\in\C^{n \times m},$ & Full-rank basis.\vspace{.3em}\\
		Output: & $\bm{p}\in\{1,\dots,n\}^m$, & Vector of row non-repeated indices.
	\end{tabular}
	\vspace{.3em}
	\begin{algorithmic}[1]
		\State [$\sim$,$\sim$,$\bm{P}$] = qr($\bm{V}_m^\ast$,`vector')
		\State $\bm{p}$ = $\bm{P}$(1:m)
	\end{algorithmic}
	\caption{QR decomposition-based Discrete Empirical Interpolation Method (Q-DEIM).}\label{alg:q-deim}
\end{algorithm}

The computational complexity of DEIM depends on the strategy for solving the linear system in line 4 of \Cref{alg:deim}, cf.~\eqref{eq:DEIM_linear_system}.
Assuming the loop over the matrix vector product $\bm{Vc}$ in line 5 of \Cref{alg:deim} to be $\mathcal{O}(nm^2)$, computing a new LU decomposition of $\bm{SV}$ from scratch in each iteration adds $\mathcal{O}(m^4)$ \cite{chaturantabut2010nonlinear} while updating the previous LU decomposition adds $\mathcal{O}(m^3)$ \cite{drmac2016new}.
Using Householder-based QR factorizations, the computational complexity of Q-DEIM is $\mathcal{O}(nm^2)$ \cite{drmac2016new}.

The cost of applying the resulting row subset selection matrix $\bm{S}$, cf.~\Cref{def:RSS_sketching_matrix}, to a vector is $\mathcal{O}(m)$.
Together with the upfront cost of DEIM or Q-DEIM, the asymptotic scaling of this approach lies above that of applying different SRFTs, which is stated as $\mathcal{O}(n\log(s))$, $\mathcal{O}(ns\log(s))$, or $\mathcal{O}(n\log(n))$ in the literature, although $s$ is typically chosen somewhat larger than $m$.
Numerical experiments in \Cref{sec:numerics_fAb} indicate that runtimes of the deterministic and randomized sketching approaches are (at least) comparable in practice.
Moreover, the row index vector $\bm{p}$ computed with DEIM or Q-DEIM is a unique property of the non-orthogonal basis $\bm{V}_m$ of the Krylov subspace $\mathcal{K}_m(\bm{A},\bm{b})$ computed with the $k$-truncated Arnoldi for fixed $m,k\in\N$.
This implies that deterministic sketching with row subset selection matrices could represent a cheaper alternative to randomized sketching with SRFTs when many applications of the sketching matrix are required for a fixed quadruple $(\bm{A},\bm{b},m,k)$.

\subsection{Over-sampling}\label{sec:over-sampling}

DEIM and Q-DEIM introduced in \Cref{sec:deim_qdeim} address the objective $\max_{\bm{S}} \sigma_{\min}(\bm{S}\bm{V}_m)$ over the family of row subset selection matrices.
By construction, both methods are limited to the sketch size $s=m$ while it is customary in randomized sketching to over-sample by, e.g., $s=2m$ or $s=4m$, cf.~\cite{nakatsukasa2024fast,guttel2023randomized}.

This subsection presents the idea underlying the two methods fast greedy missing point estimation (MPE) \cite{zimmermann2016accelerated} and Gappy proper orthogonal decomposition plus eigenvector (GappyPOD+E) \cite{peherstorfer2020stability}.
These can be used to over-sample in the deterministic row subset selection case, i.e., to extend an index vector $\bm{p}\in\{1,\dots,n\}^m$ obtained by either DEIM or Q-DEIM to $\bm{p}\in\{1,\dots,n\}^s$ with $s>m$ with the goal of further increasing $\sigma_{\min}(\bm{SV}_m)$.

Each step of both methods \cite{zimmermann2016accelerated,peherstorfer2020stability} starts from a given $\bm{S}\in\{0,1\}^{s\times n}$ and formulates the problem of selecting an additional \emph{row} $\bm{v}_+$ of $\bm{V}_m$ via a Hermitian rank-1 update of an eigenvalue problem.
With the thin singular value decomposition (SVD) $\bm{S}\bm{V}_m=\bm{U}_m\bm{\Sigma}_m\bm{W}_m^\ast$, one may write
\begin{equation*}
\begin{bmatrix}
\bm{S}\bm{V}_m\\\bm{v}_+
\end{bmatrix}
=
\begin{bmatrix}
\bm{U}_m & \bm{0}\\\bm{0}^\ast & 1
\end{bmatrix}
\begin{bmatrix}
	\bm{\Sigma}_m\\\bm{v}_+\bm{W}_m
\end{bmatrix}
\bm{W}_m^\ast
=
\begin{bmatrix}
\bm{U}_m\bm{\Sigma}_m\\\bm{v}_+\bm{W}_m
\end{bmatrix}
\bm{W}_m^\ast.
\end{equation*}
And since $\sigma_i^2(\begin{bmatrix}
\bm{S}\bm{V}_m\\\bm{v}_+
\end{bmatrix}) = \lambda_i(\begin{bmatrix}
\bm{S}\bm{V}_m\\\bm{v}_+
\end{bmatrix}^\ast\begin{bmatrix}
\bm{S}\bm{V}_m\\\bm{v}_+
\end{bmatrix})$ for all $i=1,\dots,m$, considering
$$
\begin{bmatrix}
\bm{S}\bm{V}_m\\\bm{v}_+
\end{bmatrix}^\ast\begin{bmatrix}
\bm{S}\bm{V}_m\\\bm{v}_+
\end{bmatrix} = \bm{W}_m\left(\bm{\Sigma}_m^2+(\bm{v}_+\bm{W}_m)^\ast(\bm{v}_+\bm{W}_m)\right)\bm{W}_m^\ast
$$
characterizes the singular values of $\begin{bmatrix}
\bm{S}\bm{V}_m\\\bm{v}_+
\end{bmatrix}$ as the square roots of the eigenvalues of $\left(\bm{\Sigma}_m^2+(\bm{v}_+\bm{W}_m)^\ast(\bm{v}_+\bm{W}_m)\right)$.
Weyl's inequalities for Hermitian rank-1 additions \cite[Cor.~4.3.9]{horn2012matrix} guarantee that the addition of $\bm{v}_+$ will either increase $\sigma_{\min}(\bm{S}\bm{V}_m)$ or leave it unchanged.
Fast greedy MPE and GappyPOD+E now utilize different bounds on the eigenvalues of the Hermitian rank-1 addition with the goal of choosing the row $\bm{v}_+$ that increases $\sigma_{\min}(\bm{S}\bm{V}_m)$ the most.
Since each row addition requires the computation of an SVD of size $s\times m$, the runtime of these methods is dominated by $\mathcal{O}(s^2m^2)$.
For additional details, we refer to \cite{zimmermann2016accelerated,peherstorfer2020stability}.

Interestingly, GappyPOD+E was designed to decrease the required number of rows of $\bm{V}_m$ to achieve a given value of $\sigma_{\min}(\bm{S}\bm{V}_m)$ with fewer deterministically selected rows compared to a randomized row subset selection strategy \cite{peherstorfer2020stability}.
Numerical experiments not included in this manuscript show that purely random row subset selection leads to very slow convergence in \Cref{alg:dsFOM,alg:dsGMRES,alg:dsRR}, even in the fortunate event of drawing a full-rank sketch $\bm{S}\bm{V}_m$.

Further numerical experiments not reported in this manuscript showed a very similar performance of MPE and GappyPOD+E.
In most situations, over-sampling the \mbox{Q-DEIM} row indices by MPE or GappyPOD+E by at least one row is required to obtain acceptable values of $\sigma_{\mathrm{min}}(\bm{SV}_m)$ while row indices obtained by DEIM were sometimes found to be sufficient.
\Cref{sec:numerics} reports additional empirical findings on the required degree of over-sampling.

\subsection{Adaptivity and Parallelism}

\begin{sloppypar}
Invoking fast greedy MPE or \mbox{GappyPOD+E}, cf.~\Cref{sec:over-sampling}, in addition to DEIM or Q-DEIM, cf.~\Cref{sec:deim_qdeim}, entails additional computational effort.
The choice of $s\in\N$ hence represents a trade-off between fast runtimes and large values of $\sigma_{\min}(\bm{S}\bm{V}_m)$, which translate into small subspace distortion factors $\kappa(\bm{V}_m\bm{R}_m^{-1})$, cf.~\Cref{sec:subspace_embeddings}.
\end{sloppypar}

It would therefore seem natural to choose the sketch size $s\in\N$ adaptively until some target subspace distortion factor $\kappa(\bm{V}_m\bm{R}_m^{-1})$ is reached if such preference is available.

A possible increase of computational efficiency in using DEIM, cf.~\Cref{alg:deim}, could exploit the fact that it \emph{processes} the columns of the Krylov basis in an iterative fashion.
This allows its parallelization with the $k$-truncated Arnoldi method, cf.~\Cref{alg:k_trunc_arnoldi}, which iteratively \emph{produces} the columns to be fed into DEIM.
A naive implementation using Matlab's \texttt{spmd} did not achieve speed-ups on the examples considered in \Cref{sec:numerics}.
Studying, whether a more sophisticated parallel treatment holds the potential for a reduction in runtime would be an interesting road for future research.

\section{Algorithms}\label{sec:algorithms}

This section proposes the deterministically sketched Krylov subspace methods \emph{dsFOM} for approximating the action of a matrix function on a vector, \emph{dsGMRES} for approximately solving a non-Hermitian linear system, and \emph{dsRR} for approximating a subset of the eigenvalues and -vectors of a matrix.
These are obtained as versions of the randomly sketched Krylov subspace methods sFOM \cite{guttel2023randomized}, sGMRES \cite{nakatsukasa2024fast}, and sRR \cite{nakatsukasa2024fast}, cf.~\Cref{sec:sFOM,sec:sGMRES,sec:sRR}, that replace Johnson--Lindenstrauss transforms by deterministic row subset selection sketching matrices, cf.~\Cref{sec:deterministic_sketching}.

\Cref{alg:dsFOM,alg:dsGMRES,alg:dsRR} follow the derivations in \Cref{sec:sFOM,sec:sGMRES,sec:sRR}, respectively.
In fact\footnote{with one distinction regarding the dimension of the constructed Krylov subspace discussed in the following paragraph}, sFOM, sGMRES, and sRR are recovered upon replacing lines 2-8 in \Cref{alg:dsFOM,alg:dsRR} and lines 3-9 in \Cref{alg:dsGMRES} by the set-up of a randomized Johnson--Lindenstrauss transform multiplication routine as sketching matrix $\bm{S}$.

In each algorithm, we choose the dimension of the Krylov subspace as to allow for the rigorous transfer of the error bounds \eqref{eq:sFOM_error}, \eqref{eq:sGMRES_error}, and \eqref{eq:sRR_error} from the randomly sketched methods to the deterministically sketched case.
In particular, \eqref{eq:sFOM_error} requires that $\bm{S}$ embeds $\text{Ran}(\bm{V}_m)$, which makes it sufficient to construct the deterministic row subset selection sketching matrix $\bm{S}$ for a given Krylov subspace of dimension $m$.
In contrast, \eqref{eq:sGMRES_error} requires $\bm{S}$ to embed $\text{Ran}([\bm{r}_0,\bm{AV}_m])$ while \eqref{eq:sRR_error} requires $\bm{S}$ to embed $\text{Ran}([\bm{V}_m,\bm{AV}_m])$, which is achieved by constructing $\bm{S}$ from the Krylov basis $\bm{V}_{m+1}$.
In the latter two cases, the approximations $\bm{x}_m^{\mathrm{dsGMRES}}$ and $(\lambda_i^{\mathrm{dsRR}}, \bm{x}_i^{\mathrm{dsRR}})$ for $i=1,\dots,m$ are formed from $\bm{V}_m$ while the construction of $\bm{S}$ based on the index set $\bm{p}_{m+1}$ can be seen as minimal over-sampling that either improves the embedding quality or leaves it unchanged, cf.~\Cref{sec:over-sampling}.
Then, the general subspace embedding distortion factors derived in \Cref{sec:subspace_embeddings} may be used in the place of the factors $\sqrt{\frac{1+\epsilon}{1-\epsilon}}$ and $\sqrt{\frac{1-\epsilon}{1+\epsilon}}$ in \eqref{eq:sFOM_error}, \eqref{eq:sGMRES_error}, and \eqref{eq:sRR_error} as will be discussed in the following subsections.

\Cref{alg:dsFOM,alg:dsGMRES,alg:dsRR} represent the most basic implementation of either method as a proof of concept.
Their improvement by incorporating advanced ideas such as restarting, preconditioning, or deflation is an exciting road for future research.

\subsection{dsFOM}\label{sec:dsFOM}

\begin{algorithm}[t]%[htb!]
	%	\vspace{0.3em}
	\begin{tabular}{lll}
		Input:
		& $f:\C^{n\times n} \rightarrow \C^{n\times n},$ & Matrix function.\\
		& $\bm{A}\in\C^{n \times n},$ & Matrix.\\
		& $\bm{b}\in\C^n,$ & Vector.\\
		& $m\in\N,$ & Krylov subspace dimension.\\
		& $k\in\N,$ & Truncation length for orthogonalization.\\
		& $s\in\N,$ & Sketch size.\vspace{.3em}\\
		Output: & $\bm{f}_m\in\C^n$, & dsFOM approximation to $f(\bm{A})\bm{b}$.
	\end{tabular}
	\vspace{.3em}
	\begin{algorithmic}[1]
		\State Compute basis $\bm{V}_m$ of $\mathcal{K}_m(\bm{A},\bm{b})$ by the $k$-truncated Arnoldi (\Cref{alg:k_trunc_arnoldi})
		\State Compute $\bm{p}_m\in\N^m$ by DEIM (\Cref{alg:deim}) or Q-DEIM (\Cref{alg:q-deim})
		\If{$s>m$}
		\State Starting from $\bm{p}_m$, compute $\bm{p}\in\N^s$ by fast greedy MPE or GappyPod+E (\Cref{sec:over-sampling})
		\Else
		\State $\bm{p}=\bm{p}_m$
		\EndIf
		\State Set $\bm{S} = \bm{I}(\bm{p},:)$
		\State Compute thin QR decomposition $\bm{SV}_m = \bm{Q}_m\bm{R}_m$
		\State Compute approximation $\bm{f}_m = \bm{V}_m (\bm{R}_m^{-1} f(\bm{Q}_m^\ast \bm{SAV}_m \bm{R}_m^{-1}) \bm{Q}_m^\ast \bm{Sb})$
	\end{algorithmic}
	%	\vspace{.3em}
	\caption{Deterministically sketched FOM (dsFOM).}\label{alg:dsFOM}
\end{algorithm}

\begin{sloppypar}
\Cref{alg:dsFOM} starts by generating a non-orthogonal basis \mbox{$\bm{V}_m\in\C^{n\times m}$} of the Krylov subspace \eqref{eq:krylov_subspace} by \Cref{alg:k_trunc_arnoldi} for a fixed value $k\in\N$.
While small values such as $k=2$ or $k=4$ often suffice, $k$ should be increased if the quadruple $(\bm{A},\bm{b},m,k)$ leads to an ill-conditioned or numerically rank-deficient basis $\bm{V}_m$.
A last resort for generating full-rank bases $\bm{V}_m$ for a desired value of $m$ is explicit basis whitening, i.e., replacing $\bm{V}_m$ by $\widetilde{\bm{V}}_m=\bm{V}_m\bm{R}_m^{-1}$ in \Cref{alg:k_trunc_arnoldi}, however, it has been observed that subsequently extended bases rapidly become ill-conditioned again \cite{cortinovis2024speeding}.
\end{sloppypar}

The first $m$ indices of the index vector $\bm{p}\in\{1,\dots,n\}^s$ are computed by DEIM or Q-DEIM, cf.~\Cref{sec:deim_qdeim}.
If over-sampling, i.e., $s>m$ is desired, the remaining $s-m$ row indices can be obtained by either approach outlined in \Cref{sec:over-sampling}.

After assembling the basis-specific deterministic sketching matrix $\bm{S}$ from $\bm{p}$ in line 8, cf.~\Cref{def:RSS_sketching_matrix}, (implicit) basis whitening is performed, cf.~\Cref{def:basis_whitening}, before the approximation to $f(\bm{A})\bm{b}$ derived in \Cref{sec:sFOM} is evaluated via \eqref{eq:sFOM}.
Since basis whitening guarantees $\kappa(\bm{SV}_m\bm{R}_m^{-1})=\kappa(\bm{Q}_m)=1$, the error bound \eqref{eq:sFOM_error} holds with the distortion factor $\kappa(\bm{V}_m\bm{R}_m^{-1})$ in the place of $\sqrt{\frac{1+\epsilon}{1-\epsilon}}$.

\subsection{dsGMRES}\label{sec:dsGMRES}

\begin{algorithm}[t]%[htb!]
	%	\vspace{0.3em}
	\begin{tabular}{lll}
		Input:
		& $\bm{A}\in\C^{n \times n},$ & Matrix.\\
		& $\bm{b}\in\C^n,$ & Vector.\\
		& $\bm{x}_0\in\C^n,$ & Initial guess.\\
		& $m\in\N,$ & Krylov subspace dimension.\\
		& $k\in\N,$ & Truncation length for orthogonalization.\\
		& $s\in\N,$ & Sketch size.\vspace{0.3em}\\
		Output: & $\bm{x}_m\in\C^n$, & Approximate dsGMRES solution to $\bm{Ax} = \bm{b}$.
	\end{tabular}
	\vspace{.3em}
	\begin{algorithmic}[1]
		\State Compute residual $\bm{r}_0 = \bm{b} - \bm{Ax}_0$
		\State Compute basis $\bm{V}_{m+1}$ of $\mathcal{K}_{m+1}(\bm{A},\bm{r}_0)$ and $\bm{M}_{m+1}=\bm{AV}_{m+1}$ by the $k$-truncated Arnoldi (\Cref{alg:k_trunc_arnoldi})
		\State Compute $\bm{p}_{m+1}\in\N^{m+1}$ by DEIM (\Cref{alg:deim}) or Q-DEIM (\Cref{alg:q-deim})
		\If{$s>m+1$}
		\State Starting from $\bm{p}_{m+1}$, compute $\bm{p}\in\N^s$ by fast greedy MPE or GappyPod+E (\Cref{sec:over-sampling})
		\Else
		\State $\bm{p}=\bm{p}_{m+1}$
		\EndIf
		\State Set $\bm{S} = \bm{I}(\bm{p},:)$
		\State Compute thin QR decomposition $\bm{SM}_m = \bm{Q}_m\bm{R}_m$
		\State Solve least-squares problem $\bm{y} = \bm{R}_m^{-1}(\bm{Q}_m^\ast(\bm{Sr}_0))$
		\State Compute approximation $\bm{x}_m = \bm{x}_0 + \bm{V}_m\bm{y}$
	\end{algorithmic}
	%	\vspace{.3em}
	\caption{Deterministically sketched GMRES (dsGMRES).}\label{alg:dsGMRES}
\end{algorithm}

\Cref{alg:dsGMRES} computes the initial residual $\bm{r}_0\in\C^n$ before performing the same first steps as \Cref{alg:dsFOM} for basis generation, cf.~\Cref{sec:dsFOM}.
The main difference of dsGMRES compared to dsFOM and dsRR is that in line 10 of \Cref{alg:dsGMRES}, the thin QR decomposition $\bm{SM}_m=\bm{SAV}_m=\bm{Q}_m\bm{R}_m$ is computed.
This requires the additional assumption that $\text{rank}(\bm{SAV}_m)=m$, which is satisfied since $\bm{S}$ is constructed as to extract a full-rank sketch from $\bm{V}_{m+1}$.

Since the QR decomposition $\bm{SM}_m=\bm{Q}_m\bm{R}_m$ does not correspond to basis whitening, the assumptions of \Cref{cor:whitened_embedding} are not satisfied and the subspace distortion factor $\sqrt{\frac{1+\epsilon}{1-\epsilon}}$ in the error bound \eqref{eq:sGMRES_error} needs to be replaced by the bound $\kappa(\bm{SV}_m\bm{R}_m^{-1})\kappa(\bm{V}_m\bm{R}_m^{-1})$ when performing the replacement $\widetilde{\bm{V}}_m = \bm{V}_m\bm{R}_m^{-1}$.
Interestingly, numerical experiments reported in \Cref{sec:experiments_linear_systems} suggest that neither of the two factors $\kappa(\bm{SV}_m\bm{R}_m^{-1})\kappa(\bm{V}_m\bm{R}_m^{-1})$ and $\kappa(\bm{V}_m\bm{R}_m^{-1})$ derived in \Cref{sec:subspace_embeddings} provides informative upper bounds.
The analysis of this behavior is an interesting question for future work.

\Cref{alg:dsGMRES} terminates by approximating the solution to $\bm{Ax}=\bm{b}$ via \eqref{eq:sGMRES} derived in \Cref{sec:sGMRES}.

\subsection{dsRR}\label{sec:dsRR}

\begin{algorithm}[t]%[htb!]
%	\vspace{0.3em}
	\begin{tabular}{lll}
		Input:
		& $\bm{A}\in\C^{n \times n},$ & Matrix.\\
		& $\bm{b}\in\C^n,$ & Starting vector.\\
		& $m\in\N,$ & Krylov subspace dimension.\\
		& $k\in\N,$ & Truncation length for orthogonalization.\\
		& $s\in\N,$ & Sketch size.\vspace{0.3em}\\
		Output: & $\lambda_i\in\C, i=1,\dots,m,$ & Approximate eigenvalues of $\bm{A}$.\\
		& $\bm{x}_i\in\C^n, i=1,\dots,m,$ & Approximate normalized eigenvectors of $\bm{A}$.
	\end{tabular}
	\vspace{.3em}
	\begin{algorithmic}[1]
		\State Compute basis $\bm{V}_{m+1}$ of $\mathcal{K}_{m+1}(\bm{A},\bm{b})$ by the $k$-truncated Arnoldi (\Cref{alg:k_trunc_arnoldi})
		\State Compute $\bm{p}_{m+1}\in\N^{m+1}$ by DEIM (\Cref{alg:deim}) or Q-DEIM (\Cref{alg:q-deim})
		\If{$s>m+1$}
		\State Starting from $\bm{p}_{m+1}$, compute $\bm{p}\in\N^s$ by fast greedy MPE or GappyPod+E (\Cref{sec:over-sampling})
		\Else
		\State $\bm{p}=\bm{p}_{m+1}$
		\EndIf
		\State Set $\bm{S} = \bm{I}(\bm{p},:)$
		\State Compute thin QR decomposition $\bm{SV}_m = \bm{Q}_m\bm{R}_m$
		%		\If{$\kappa(\bm{R}_m) > 1/\mathtt{eps}$}
		%			\State Whiten the basis via $\bm{V}_m = \bm{V}_m\bm{R}_m^{-1}$
		%			\State Compute new thin QR decomposition $\bm{SV}_m = \bm{Q}_m\bm{R}_m$
		%		\EndIf
		\State Solve eigenvalue problem $(\bm{R}_m^{-1} \bm{Q}_m^\ast (\bm{SAV}_m)) \bm{y}_i = \lambda_i\bm{y}_i$ for $i=1,\dots,m$
		\State Compute $\bm{x}_i=\bm{V}_m\bm{y}_i / \|\bm{V}_m\bm{y}_i\|_2$ for $i=1,\dots,m$
	\end{algorithmic}
%	\vspace{.3em}
	\caption{Deterministically sketched Rayleigh--Ritz (dsRR).}\label{alg:dsRR}
\end{algorithm}

\begin{sloppypar}
The first 9 lines of \Cref{alg:dsRR} correspond to those of \Cref{alg:dsFOM} with the distinction that the Krylov subspace dimension $m$ is replaced by $m+1$.
As discussed previously, the $m+1$-dimensional basis is used for generating the sketching matrix $\bm{S}$ while the approximate eigenpairs are computed based on the $m$-dimensional basis.
Solving the full eigenvalue problem of the matrix \mbox{$\bm{R}_m^{-1} \bm{Q}_m^\ast (\bm{SAV}_m)\in\C^{m\times m}$} derived in \Cref{sec:sRR} can be achieved in $\mathcal{O}(m^3)$ runtime by standard methods such as the QR algorithm \cite{golub2013matrix}.
\end{sloppypar}

Having normalized the approximate eigenvectors $\bm{x}_i\in\C^n$ in line 11 of \Cref{alg:dsRR}, we restate the error bound \eqref{eq:sRR_error} in the setting of \Cref{sec:whitened_basis}.
The two inequalities obtained from \eqref{eq:whitened_subspace_embedding} for $\bm{v}=\bm{Ax}_i-\lambda_i\bm{x}_i$ yield the error bounds
\begin{equation}\label{eq:dsRR_error}
\sigma_{\mathrm{min}}(\bm{V}_m\bm{R}_m^{-1}) \|\bm{S}(\bm{Ax}_i-\lambda_i\bm{x}_i)\|_2 \leq \|\bm{Ax}_i-\lambda_i\bm{x}_i\|_2 \leq \sigma_{\mathrm{max}}(\bm{V}_m\bm{R}_m^{-1}) \|\bm{S}(\bm{Ax}_i-\lambda_i\bm{x}_i)\|_2.
\end{equation}

\section{Numerical experiments}\label{sec:numerics}

\begin{sloppypar}
In this section, dsFOM (\Cref{alg:dsFOM}), \mbox{dsGMRES} (\Cref{alg:dsGMRES}), and dsRR (\Cref{alg:dsRR}) are each tested on one example problem and compared against their unsketched and randomly sketched counterparts.
Throughout the experiments, randomized methods use discrete cosine transforms (DCT) as sketching matrices, MPE uses DEIM and GappyPOD+E uses Q-DEIM as initialization.
Runtimes were recorded with Matlab R2025b on an Intel i5-1135G7 processor with 4 cores, 2.40 GHz, and 8 GB RAM memory.
Matlab implementations are available under \url{https://github.com/KBergermann/dsKrylov}.
\end{sloppypar}

\subsection{Matrix functions}\label{sec:numerics_fAb}

We consider an example of exponential integrators, which are well-suited to integrate stiff or highly-oscillatory ordinary differential equations \cite{hochbruck2010exponential}.
We define the semi-discretized semi-linear parabolic differential equation
$$
\bm{u}'(t) = D\bm{Lu}(t) + g(\bm{u}(t)), \quad \bm{u}(0) = \bm{u}_0,
$$
in the unknown $\bm{u}: [0,T] \rightarrow \C^n$ with $D>0$ a diffusion constant, $\bm{L}\in\R^{n\times n}$ the symmetric negative semi-definite finite difference Neumann Laplacian on an equispaced grid of the spatial domain $[-1,1]^2$, and $g:\C^n\rightarrow\C^n$ a non-linear function.
In particular, we choose $D=\frac{1}{40}$, $g(\bm{u})=\frac{1}{4}\bm{u}(1-\bm{u})$, and the initial conditions $\bm{u}_0\in\C^n$ to originate from evaluating the function $h(x,y)=\frac{1}{2}e^{-x^2}e^{-y^2}$ on the spatial grid.

We consider the exponential Euler method \cite{hochbruck2010exponential}, exploiting a result by Al-Mohy and Higham \cite{al2011computing} that has recently been used in Krylov methods \cite{gaudreault2018kiops,bergermann2024adaptive}.
Specifically, one step of the exponential Euler with time step size $t=1$ is obtained by evaluating
$$
e^{\bm{A}}\bm{b} = e^{D\bm{L}}\bm{u}_0 + \varphi_1(D\bm{L})g(\bm{u}_0), \quad \bm{A} = \begin{bmatrix}
D\bm{L} & g(\bm{u}_0)\\
\bm{0}^\ast & 0
\end{bmatrix}, \quad \bm{b} = \begin{bmatrix}
\bm{u}_0\\ 1
\end{bmatrix}, \quad \varphi_1(z)=\frac{e^z-1}{z},
$$
and extracting the first $n$ entries of the result $e^{\bm{A}}\bm{b}$.
Note that $\bm{A}\in\C^{(n+1)\times (n+1)}$ is non-symmetric although $\bm{L}$ is symmetric.
We choose $n=256^2$, the maximal Krylov subspace dimension $m=260$, and truncation parameter $k=2$ in \Cref{alg:k_trunc_arnoldi}, which leads to $\kappa(\bm{V}_m)\approx 2.09\cdot 10^3$.
The condition number of the basis is heavily initial condition-dependent and may reach inverse machine precision for other choices.

Since no notion of a residual exists for $f(\bm{A})\bm{b}$, we compute a reference solution via the classical FOM approximation \eqref{eq:FOM} using an orthogonal Krylov basis of dimension $m=350$ for which an a-posteriori error estimate indicates accuracy up to machine precision \cite{saad1992analysis}.

\begin{figure}
	\centering
	\subfloat[Absolute error $\|\bm{f}_m^\star - e^{\bm{A}}\bm{b}\|_2$]{
		\includegraphics[width=0.493\textwidth]{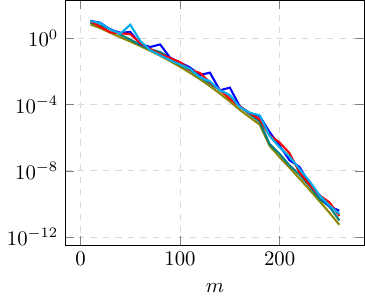}\label{fig:fom:abs}
	}
	\hfill
	\subfloat[Relative error $\frac{\|\bm{f}_m^\star - e^{\bm{A}}\bm{b}\|_2}{\|\bm{f}_m^{\mathrm{FOM}} - e^{\bm{A}}\bm{b}\|_2}$]{
		\includegraphics[width=0.467\textwidth]{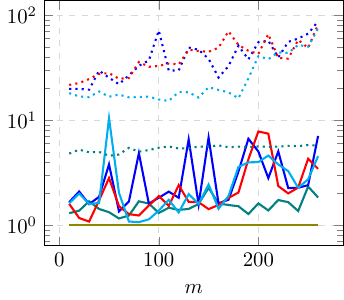}\label{fig:fom:rel}
	}
	\vspace{5pt}
	\centering
	\includegraphics[width=0.9\textwidth]{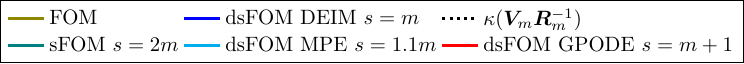}
	\vspace{-10pt}
	\caption{Approximation errors of FOM, sketched FOM (sFOM), and deterministically sketched FOM (dsFOM) to $e^{\bm{A}}\bm{b}$ with $\bm{A}$ a 2d finite difference Laplacian with one appended row and column.
	The symbol $\star\in\{\mathrm{FOM}, \mathrm{sFOM}, \mathrm{dsFOM}\}$ holds the place for different approximations, e.g., \eqref{eq:FOM} and \eqref{eq:sFOM}.
	For dsFOM, the three row subset selection strategies DEIM ($s=m$), GPODE (GappyPOD+E, $s=m+1$), and MPE ($s=1.1m$) are employed.
	(a) Absolute errors are plotted against the Krylov dimension $m$, (b) relative errors alongside condition number error bounds are plotted against $m$.
	The reference solution $e^{\bm{A}}\bm{b}$ is computed with FOM with sufficiently many iterations to guarantee high accuracy.}\label{fig:fom}
\end{figure}

\Cref{fig:fom:abs} shows that absolute errors $\|\bm{f}_m - e^{\bm{A}}\bm{b}\|_2$ of FOM, sFOM, and dsFOM with different row subset selection strategies behave similarly with respect to $m$.
Using the row indices obtained by Q-DEIM to construct the deterministic sketching matrix for dsFOM leads to substantial approximation errors that are indicated by values of the subspace distortion factor $\kappa(\bm{V}_m\bm{R}_m^{-1})$ of above $10^{10}$, cf.~\Cref{cor:whitened_embedding}.
Although current analyses of sFOM consider different quantities, cf.~\eqref{eq:sFOM_error} and \cite{cortinovis2024speeding,palitta2025sketched}, \Cref{fig:fom:rel} suggests that $\kappa(\bm{V}_m\bm{R}_m^{-1})$ may also take the role of a bound on the relative error $\frac{\|\bm{f}_m - e^{\bm{A}}\bm{b}\|_2}{\|\bm{f}_m^{\mathrm{FOM}} - e^{\bm{A}}\bm{b}\|_2}$ of sFOM and dsFOM with respect to the orthogonal FOM approximation \eqref{eq:FOM}.
In particular, $\kappa(\bm{V}_m\bm{R}_m^{-1})$ for sFOM ranging between $5$ and $6$ computationally confirms \Cref{rem:JL_distortion}.

The described dissatisfactory behavior of Q-DEIM can be cured by adding one additional row using GappyPOD+E with $s=m+1$.
Over-sampling the DEIM row indices by MPE with $s=1.1m$ shows some improvements in terms of approximation errors and the distortion factor $\kappa(\bm{V}_m\bm{R}_m^{-1})$, cf.~\Cref{fig:fom:rel}.

On top of the $2$-truncated Arnoldi ($0.489$s), runtimes of the remainder of \Cref{alg:dsFOM} for dsFOM with DEIM ($10.5$s), MPE with $s=1.1m$ ($15.4$s), and GappyPOD+E with $s=m+1$ ($7.94$s) range between those of sFOM with DCT ($1.19$s) and the fash Walsh--Hadamard transform ($20.7$s) as sketching matrices with $s=2m$ for the parameters $n=65\,536$ and $m=260$.
The unfavorable asymptotic runtime dependence of row subset selection methods on the Krylov dimension $m$ discussed at the end of \Cref{sec:deim_qdeim} currently prevents this approach to be competitive with DCT in sFOM since relatively few applications of the sketching matrix are required.

\subsection{Linear systems}\label{sec:experiments_linear_systems}

\begin{figure}
	\centering
	\subfloat[Absolute residual $\|(\bm{I}-\bm{A})\bm{x}_m^\star - \bm{b}\|_2$]{
		\includegraphics[width=0.98\textwidth]{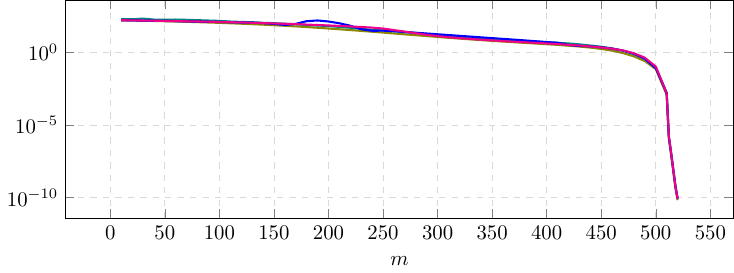}\label{fig:GMRES:abs}
	}
	
	\subfloat[Relative residual $\frac{\|(\bm{I}-\bm{A})\bm{x}_m^\star - \bm{b}\|_2}{\|(\bm{I}-\bm{A})\bm{x}_m^{\mathrm{GMRES}} - \bm{b}\|_2}$]{
		\includegraphics[width=0.48\textwidth]{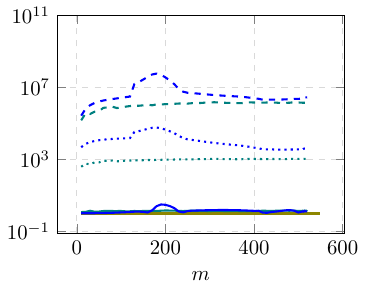}\label{fig:GMRES:rel_DEIM}
	}
	\hfill
	\subfloat[Relative residual $\frac{\|(\bm{I}-\bm{A})\bm{x}_m^\star - \bm{b}\|_2}{\|(\bm{I}-\bm{A})\bm{x}_m^{\mathrm{GMRES}} - \bm{b}\|_2}$]{
		\includegraphics[width=0.48\textwidth]{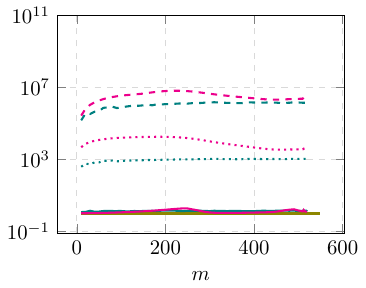}\label{fig:GMRES:rel_GPODE}
	}
	\vspace{5pt}
	\centering
	\includegraphics[width=0.75\textwidth]{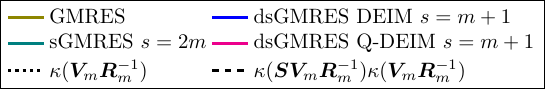}
	\vspace{-10pt}
	\caption{Residuals of GMRES, sketched GMRES (sGMRES) and deterministically sketched GMRES (dsGMRES) for $(\bm{I}-\bm{A})\bm{x}=\bm{b}$ with $\bm{A}$ a convection-diffusion operator.
	The symbol \mbox{$\star\in\{\mathrm{GMRES}, \mathrm{sGMRES}, \mathrm{dsGMRES}\}$} holds the place for different approximations, e.g., \eqref{eq:GMRES} and \eqref{eq:sGMRES}.
	For dsGMRES, the two row subset selection strategies DEIM ($s=m+1$) and Q-DEIM ($s=m+1$) are employed.
	(a) Absolute residuals for all four methods are plotted against the Krylov dimension $m$, (b)-(c) relative residuals alongside two condition number bounds are plotted against $m$.}\label{fig:gmres}
\end{figure}

We consider the non-symmetric linear initial value problem
$$
\bm{u}'(t) = \bm{Au}(t), \quad \bm{u}(0) = \bm{u}_0,
$$
in the unknown $\bm{u}: [0,T] \rightarrow \C^n$ with $\bm{A}\in\R^{n \times n}$ a semi-discretized negative definite convection-diffusion operator in the convection-dominated regime as considered in \cite[Sec.~5.1]{guttel2023randomized}.
More specifically, we define $\bm{A}=D\bm{L}+\bm{C}$ on $d\in\N$ equispaced grid points in both dimensions of the spatial domain $[0,1]^2$ with $D=10^{-3}$, $\bm{L}=(d-1)^2(\widetilde{\bm{L}}\otimes\bm{I}+\bm{I}\otimes\widetilde{\bm{L}})$ with $\widetilde{\bm{L}}=\text{tridiag}(1,-2,1)\in\R^{d\times d}$ the 1d Laplacian, and $\bm{C}=(d-1)(\widetilde{\bm{C}}\otimes\bm{I}+\bm{I}\otimes\widetilde{\bm{C}}^\ast)$ with $\widetilde{\bm{C}}=\text{tridiag}(1,-1,0)\in\R^{d\times d}$ a 1d convection operator.

\begin{sloppypar}
	One step of the implicit Euler method with $t=1$ leads to the linear system \mbox{$(\bm{I}-\bm{A})\bm{x}=\bm{b}$} with $\bm{b}=\bm{u}_0$, which is constructed by evaluating the function \mbox{$h(x,y)=0.3+256xy(1-x)(1-y)$} on the spatial grid.
	Choosing \mbox{$n=d^2=256^2=65\,536$}, the maximal Krylov subspace dimension $m=520$, and truncation parameter $k=4$ in \Cref{alg:k_trunc_arnoldi} leads to $\kappa(\bm{V}_m)=4.13\cdot 10^{12}$.
\end{sloppypar}

\Cref{fig:GMRES:abs} shows a very similar residual convergence behavior of (plain) GMRES, sGMRES, and dsGMRES with different row subset selection strategies.
In \mbox{dsGMRES}, the observation from \Cref{sec:numerics_fAb} that Q-DEIM requires minimal over-sampling is cured by the fact that it is constructed from the Krylov basis $\bm{V}_{m+1}$, cf.~\Cref{fig:GMRES:rel_GPODE}.
In this example, over-sampling does not offer a significant improvement of the convergence behavior of dsGMRES.

Since sGMRES and dsGMRES contain no basis whitening, the error bound \eqref{eq:sGMRES_error} together with the analysis from \Cref{sec:subspace_embeddings} suggest the subspace distortion factor \eqref{eq:distortion_cond_numbers} after basis transformation via $\bm{R}_m^{-1}$, i.e., $\kappa(\bm{SV}_m\bm{R}_m^{-1})\kappa(\bm{V}_m\bm{R}_m^{-1})$, as an error bound on the relative residual $\frac{\|(\bm{I}-\bm{A})\bm{x}_m - \bm{b}\|_2}{\|(\bm{I}-\bm{A})\bm{x}_m^{\mathrm{GMRES}} - \bm{b}\|_2}$.
The dashed lines in \Cref{fig:GMRES:rel_DEIM,fig:GMRES:rel_GPODE} illustrate that this bound is impractically large.
Similarly, the condition number $\kappa(\bm{V}_m\bm{R}_m^{-1})$ that \emph{would apply if sGMRES and dsGMRES were to use basis whitening} is not sharp in this case either.
However, the relative convergence history of the dsGMRES DEIM error in \Cref{fig:GMRES:rel_DEIM} is still somewhat reflected in the condition number bounds.
A deeper investigation of this behavior would be an interesting future research direction.

\subsection{Eigenvalue problems}

Finally, we consider a problem from network science.
We choose the directed \texttt{web-Stanford} graph\footnote{\url{https://sparse.tamu.edu/SNAP/web-Stanford}} \cite{leskovec2009community}, where nodes represent websites and edges hyperlinks.
The graph is unweighted, has no self-loops, and $n=257\,824$ nodes after removing all nodes with zero in-degrees.
We use the resulting non-symmetric adjacency matrix $\bm{A}\in\R^{n\times n}$ to construct the normalized graph in-Laplacian $\bm{L} = \bm{I} - \bm{D}_{\mathrm{in}}^{-1/2}\bm{A}\bm{D}_{\mathrm{in}}^{-1/2}$, with $\bm{D}_{\mathrm{in}} = \text{diag}(\bm{A}^\ast\bm{1})$, whose extremal eigenpairs contain information on structural and dynamical network properties \cite{bonacich1987power,pons2005computing,von2007tutorial,bergermann2025gradient}.

\begin{figure}
	\centering
	\subfloat[Residuals $\|\bm{Lx}^\star_i-\lambda_i^\star\bm{x}^\star_i\|_2$ for $i=1,\dots,m$]{
		\includegraphics[width=0.98\textwidth]{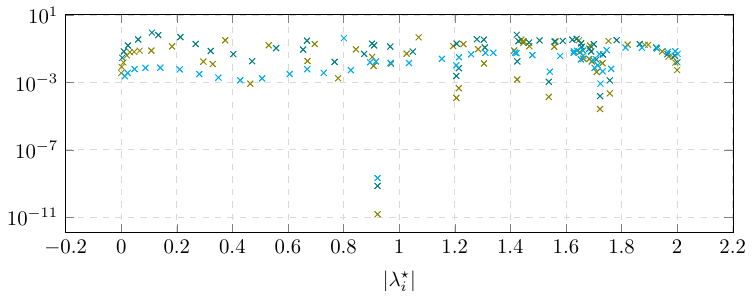}\label{fig:RR:scatter}
	}
	
\subfloat[Residual $\|\bm{Lx}^\star_2-\lambda_2^\star\bm{x}^\star_2\|_2$ of the Fiedler vector]{
		\includegraphics[width=0.498\textwidth]{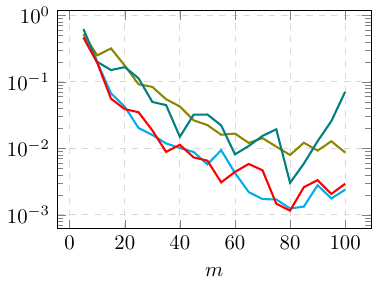}\label{fig:RR:convergence}
}
\hfill
%\raisebox{-6pt}{
	\subfloat[Relative residuals $\frac{\|\bm{Lx}^\star_2-\lambda_2^\star\bm{x}^\star_2\|_2}{\|\bm{S}(\bm{Lx}^\star_2-\lambda_2^\star\bm{x}_2^\star)\|_2}$ and their bounds]{
		\includegraphics[width=0.448\textwidth]{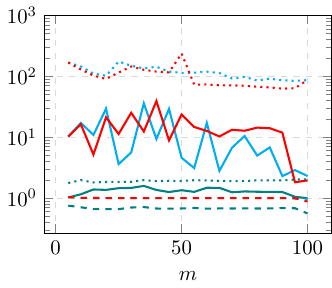}\label{fig:RR:bounds}
}
%\vspace{5pt}

	\vspace{5pt}
	\centering
	\includegraphics[width=0.8\textwidth]{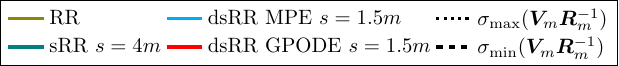}
	\caption{Eigenvector residuals of Rayleigh--Ritz (RR), sketched RR (sRR), and deterministically sketched RR (dsRR) to the eigenvalue problem $\bm{Lx}=\lambda\bm{x}$ with $\bm{L}$ a normalized non-symmetric graph in-Laplacian.
	The symbol $\star\in\{\mathrm{RR}, \mathrm{sRR}, \mathrm{dsRR}\}$ holds the place for different approximations, cf.~\Cref{sec:sRR}.
	For dsRR, the two row subset selection strategies MPE and GPODE (GappyPOD+E) are employed with sketch size $s=1.5m$.
	(a) All $m=100$ eigenvector residuals are plotted against the respective eigenvalue magnitude, (b) convergence of the Fiedler vector residual is plotted against the Krylov dimension $m$, (c) relative residuals and their bounds are plotted against $m$.}
\end{figure}

Constructing a non-orthogonal Krylov basis with uniformly random starting vector $\bm{b}$ via \Cref{alg:k_trunc_arnoldi} with $k=10$ and $m=100$ leads to $\kappa(\bm{V}_m)=5.04\cdot 10^{15}$.
\Cref{fig:RR:scatter} shows that eigenvector residuals obtained by dsRR MPE with $s=1.5m$ are similar to those of RR and sRR with $s=4m$ for large magnitude eigenpairs and somewhat lower for small magnitude eigenpairs.
The latter are of major interest as they decode community structure of the network.
In particular, \Cref{fig:RR:convergence} shows that convergence of the residual of the Fiedler vector $\bm{x}_2$ in the dsRR variants is faster compared to RR and sRR.
\Cref{fig:RR:bounds} illustrates the satisfactory behavior of the error bounds \eqref{eq:dsRR_error}.

Similarly to sRR, a larger degree of over-sampling in comparison to dsFOM and dsGMRES is advisable for dsRR.
In particular, we also observe the occurrence of spurious Ritz values outside of the field-of-values of the original matrix discussed in \cite{guttel2023randomized,palitta2025sketched} in the deterministically sketched case.
Numerical experiments not included in this manuscript suggest the tendency of these spurious Ritz values to move closer to the field-of-values as the degree of over-sampling is increased.

\section{Conclusion and outlook}\label{sec:conclusion}

This manuscript proposes a new class of sketched Krylov subspace methods for matrix functions, linear systems, and eigenvalue problems.
Leveraging new theoretical insights from the analysis of subspace embeddings obtained by arbitrary sketching matrices, they utilize deterministic row subset selection matrices instead of randomized Johnson--Lindenstrauss transforms as sketching matrices.
They achieve similar accuracies to their randomly sketched counterparts, but construct subspace embeddings that hold with probability 1.

Several parts of the manuscript highlight open questions outside of its scope.
We presented dsFOM, dsGMRES, and dsRR in their most basic form.
Future work could study the incorporation of advanced techniques such as restarting, preconditioning, deflation, other basis generation techniques, or the development of different deterministic sketching approaches that might be faster or equipped with sharper a-priori bounds.

Randomized sketching with oblivious subspace embeddings represents a powerful and generally applicable framework.
However, \Cref{fig:trace_estimation} uses the example of trace estimation to illustrate that derandomized approaches exploiting specific problem properties hold the potential to show superior performance.
This manuscript showcases the efficacy of deterministic sketching in Krylov subspace methods.
A future direction in sketched numerical linear algebra could be the development of deterministic problem-specific sketching approaches that outperform randomized methods by exploiting problem-specific properties.

\section*{Acknowledgments}

The author thanks Alice Cortinovis, Stefano Massei, and Marcel Schweitzer for helpful discussions.

\appendix

\section{Background from model order reduction}\label{sec:MOR}

The discrete empirical interpolation method (DEIM) \cite{chaturantabut2010nonlinear}, cf.~\Cref{alg:deim} and \Cref{sec:deim_qdeim}, was originally introduced in the context of model order reduction.
Here, a prototypical problem is that of having to solve large-scale time- or parameter-dependent non-linear system of ordinary differential equations
\begin{equation}\label{eq:MOR_ODE}
\frac{\partial\bm{y}(t)}{\partial t} = \bm{Ay}(t) + F(\bm{y}(t)), \qquad \bm{Ay}(\mu) + F(\bm{y}(\mu)) = \bm{0}
\end{equation}
with $\bm{y}:\R\supset\mathcal{I}\rightarrow\C^n, \bm{A}\in\C^{n\times n},$ and $F:\C^n\rightarrow\C^n$ non-linear for many time or parameter values.
A successful approach for reducing the complexity of \eqref{eq:MOR_ODE} is dimensionality reduction via Galerkin projection \cite{benner2015survey,chaturantabut2010nonlinear}.
Here, a given orthonormal basis $\bm{U}_k\in\C^{n\times k}$ is used to obtain the reduced-order systems
\begin{equation}\label{eq:MOR_reduced}
\frac{\partial\widetilde{\bm{y}}(t)}{\partial t} = \bm{U}_k^\ast\bm{A}\bm{U}_k\widetilde{\bm{y}}(t) + \bm{U}_k^\ast F(\bm{U}_k\widetilde{\bm{y}}(t)), \qquad \bm{U}_k^\ast\bm{A}\bm{U}_k\widetilde{\bm{y}}(\mu) + \bm{U}_k^\ast F(\bm{U}_k\widetilde{\bm{y}}(\mu)) = \bm{0}.
\end{equation}
in a new variable $\widetilde{\bm{y}}:\mathcal{I}\rightarrow\C^k$ defined via $\bm{y} = \bm{U}_k\widetilde{\bm{y}}$.
Such a basis $\bm{U}_k\in\C^{n\times k}$ is often constructed from the proper orthogonal decomposition (POD) \cite{benner2015survey}.
Its idea is to solve the expensive problem \eqref{eq:MOR_ODE} for a small number of time or parameter values, collect the solutions $\bm{y}_1,\dots,\bm{y}_{n_s}$ as columns of a snapshot matrix $\bm{Y}\in\C^{n\times n_s}$, and compute the thin singular value decomposition
$$
\bm{Y} = [\bm{y}_1, \dots,\bm{y}_{n_s}] = \bm{U}_k\bm{\Sigma}_k\bm{W}_k^\ast,
$$
which is well-known to be the best rank-$k$ approximation of $\bm{Y}$.
The orthonormal basis used in the Galerkin projection then corresponds to the matrix $\bm{U}_k$ of left singular vectors of $\bm{Y}$.

The major computational effort in interacting with the reduced-order model \eqref{eq:MOR_reduced} lies in evaluating the non-linearity
\begin{equation}\label{eq:MOR_non-linearity}
\bm{f}(\tau) = \begin{cases}
F(\bm{U}_k\widetilde{\bm{y}}(t)), & \text{time-dependent,}\\
F(\bm{U}_k\widetilde{\bm{y}}(\mu)), & \text{parameter-dependent,}
\end{cases}
\end{equation}
since its arguments are still elements of $\C^n$ \cite{chaturantabut2010nonlinear}.
This issue is known as the lifting bottleneck.

The discrete empirical interpolation method (DEIM) \cite{chaturantabut2010nonlinear} addresses this by applying the projection
\begin{equation}\label{eq:DEIM_projection}
\bm{f}(\tau) = \widehat{\bm{V}}_m\bm{c}(\tau)
\end{equation}
via some full-rank basis $\widehat{\bm{V}}_m\in\C^{n\times m}$.
In order to uniquely determine the coefficient vector $\bm{c}\in\C^m$, DEIM selects $m$ linearly independent rows from \eqref{eq:DEIM_projection}.
In the notation of \Cref{def:RSS_sketching_matrix}, this leads to the linear system
\begin{equation}\label{eq:DEIM_linear_system}
\bm{S}\bm{f}(\tau) = (\bm{S}\widehat{\bm{V}}_m)\bm{c}(\tau) \Leftrightarrow  \bm{c}(\tau) = (\bm{S}\widehat{\bm{V}}_m)^{-1}\bm{S}\bm{f}(\tau).
\end{equation}
Similarly to the POD approach for the full system \eqref{eq:MOR_ODE}, the basis $\widehat{\bm{V}}_m$ is obtained from the thin singular value decomposition of the non-linear snapshots $F(\bm{Y}) = \widehat{\bm{V}}_m\widehat{\bm{\Sigma}}_m\widehat{\bm{W}}_m$.
Here, the non-linearity $F$ is applied column-wise to the previously computed snapshots $\bm{Y} = [\bm{y}_1, \dots,\bm{y}_{n_s}]$.

\end{document}